\documentclass[10pt]{amsart}

\usepackage{a4wide,amsmath,amssymb}
\usepackage[toc,page]{appendix}
\usepackage{hyperref}

\usepackage{url}

\usepackage{courier}
\usepackage{easyReview}

\usepackage{mathtools}

\usepackage{xcolor}

\newtheorem{thm}{Theorem}[section]
\newtheorem{lemma}[thm]{Lemma}
\newtheorem{prop}[thm]{Proposition}
\newtheorem{cor}[thm]{Corollary}
\theoremstyle{remark}
\newtheorem{rem}[thm]{Remark}

\newtheorem{ex}[thm]{Example}

\theoremstyle{definition}
\newtheorem{defn}[thm]{Definition}
\newtheoremstyle{Claim}{}{}{\itshape}{}{\itshape\bfseries}{:}{ }{#1}
\theoremstyle{Claim}

\newcommand{\T}{{\mathbb{T}}}
\newcommand{\Z}{{\mathbb{Z}}}

\newcommand{\R}{\mathbb{R}}

\newcommand{\eps}{\varepsilon}

\DeclareMathOperator{\esssup}{ess\,sup}

\theoremstyle{plain}

\newcommand{\dx}{\Delta x}
\newcommand{\dy}{\Delta y}
\newcommand{\dt}{\Delta t}
\newcommand{\D}{\Delta}
\newcommand{\normdiscrete}[1]{\| #1\|_{W^2(\mathcal{M})}}
\makeatletter
\@namedef{subjclassname@2020}{\textup{2020} Mathematics Subject Classification}
\makeatother

\begin{document}

\title[Quantitative and qualitative properties for Hamilton-Jacobi PDEs]{Quantitative and qualitative properties for Hamilton-Jacobi PDEs via the nonlinear adjoint method}

\author{Fabio Camilli}
\address{SBAI, Sapienza Universit\`{a} di Roma,
	via A.Scarpa 14, 00161 Roma (Italy)}
\curraddr{}
\email{fabio.camilli@uniroma1.it}

\author{Alessandro Goffi}
\address{Dipartimento di Matematica e Informatica ``Ulisse Dini'', Universit\`a degli Studi di Firenze, 
viale G. Morgagni 67/A, 50134 Firenze (Italy)}
\curraddr{}
\email{alessandro.goffi@unifi.it}

\author{Cristian Mendico}
\address{ Institut de Math\'ematique de Bourgogne, UMR 5584 CNRS, Universit\'e Bourgogne, 21000 Dijon, France}
\curraddr{}
\email{cristian.mendico@u-bourgogne.fr}

 \thanks{
The authors wish to thank Martino Bardi, Jeff Calder, Alekos Cecchin, Samuel Daudin, Lawrence Craig Evans, Philippe Souplet and Eitan Tadmor for several pointers to  references and   fruitful comments on the material of the paper. A. Goffi was partially supported by the INdAM-GNAMPA projects 2022 and 2023. C. Mendico was partially supported by the INdAM-GNAMPA projects 2022 and 2023, by the MIUR Excellence Department Project awarded to the Department of Mathematics, University of Rome Tor Vergata, CUP E83C23000330006, and the project E83C22001720005 “ConDiTransPDE”, Control, diffusion and transport problems in PDEs and applications. F. Camilli and A. Goffi were partially supported by the project funded by the EuropeanUnion – NextGenerationEU under the National Recovery and Resilience Plan (NRRP), Mission 4 Component 2 Investment 1.1 - Call PRIN 2022 No. 104 of February 2, 2022 of Italian Ministry of University and Research; Project 2022W58BJ5 (subject area: PE - Physical Sciences and Engineering) ``PDEs and optimal control methods in mean field games, population dynamics and multi-agent models".  The authors were partially supported by the King Abdullah University of Science and Technology (KAUST) project CRG2021-4674 ``Mean-Field Games: models, theory and computational aspects"}



\date{\today}

\begin{abstract}
 We provide some new integral estimates for solutions  to 
 Hamilton-Jacobi equations and we discuss several consequences, ranging from $L^p$-rates of convergence for the vanishing viscosity approximation to regularizing effects for the Cauchy problem in the whole Euclidean space and Liouville-type theorems.  Our approach is based on duality techniques \`a la Evans and a careful study of   advection-diffusion equations.  The optimality of the results is discussed by several examples.
\end{abstract}

\maketitle

\section{Introduction}
The powerful theory of viscosity solutions, by means of the maximum principle, allows to obtain general existence, uniqueness and regularity results for the first-order Hamilton-Jacobi equation 
\begin{equation}\label{hjintroinv}
	\begin{cases}
		\partial_t u+H(x ,Du)=f(x,t)&\text{ in } Q_T:=M\times(0,T),\\
		u(x,0)=u_0(x)&\text{ in }M.
	\end{cases}
\end{equation}
(see \cite{BCD,Barles,CL83tams,CEL}).
Nonetheless, a recent approach known as the {\em nonlinear adjoint method} was developed by L.C. Evans in \cite{Evansadjoint}  (see also \cite{K65,LinTadmor} for related results) to capture finer properties of the solution $u$ for a nonconvex  Hamiltonian $H$ left open by the viscosity theory. The core idea of this strategy is the study of  the linear (backward) problem
\begin{equation}\label{adjointintro}
	\begin{cases}
		-\partial_{t} \rho - \eps \Delta \rho - \textrm{div} (D_pH(x, Du_\eps) \rho) = 0 &\text{ in }Q_\tau:=M\times(0,\tau),
		\\
		\rho(x, \tau) = \rho_{\tau}(x) &\text{ in }M,
	\end{cases}
\end{equation}
where $\tau\in(0,T)$ and
\[
\rho_\tau\geq0\text{ and }\int_{M}\rho_\tau(x)\,dx=1,
\]
which is the adjoint problem of the linearized of the  viscous regularization of \eqref{hjintroinv}, i.e.
\begin{equation}\label{hjintro}
	\begin{cases}
		\partial_t u_\eps-\eps\Delta u_\eps+ H(x,Du_\eps)=f(x,t)&\text{ in }Q_T,\\
		u_\varepsilon(x,0)=u_0(x)&\text{ in }M.
	\end{cases}
\end{equation}
It makes little use of the maximum principle (only to show that $\rho\geq0$) and provides a better understanding of the convergence $u_\eps$ to $u$, along with the gradient shock structures of the solution $u$. \\
Here, for various choices of the regularity properties of the terminal condition $\rho_\tau$, we derive estimates for the solution $\rho$ to \eqref{adjointintro} to retrieve useful properties of $u_\eps$, sometimes stable with respect to the viscosity parameter $\eps$. We confine our analysis on $M=\R^n$ or the flat torus $M=\T^n$. In the latter case the data of the problem are assumed $\Z^n$-periodic with respect to the space variable. \\ 
We first exploit  this approach to investigate some second-order regularity properties of solutions to Hamilton-Jacobi equations in $L^p$-spaces, $1\leq p\leq\infty$. Examples of them are the one-side bounds
\begin{equation}\label{semicintro}
u_{ee}\leq (u_0)_{ee}+\int_0^Tc_f(t)\,dt,\quad\text{for } e\in\R^n,\ |e|=1,\ f_{ee}\leq c_f(t)\in L^1(0,T)
\end{equation}
and  
\[
\Delta u\leq C(x,t)\in L^1_t(L^p_x),\, 1\leq p\leq\infty ,
\]
which are motivated by uniqueness and stability properties initiated in \cite{LinTadmor,L82book}. As a byproduct, we address the following issues, referring to the introductory part within each section for a detailed comparison and the improvements with respect to the literature:
\begin{itemize}
	\item Convergence rates for the vanishing viscosity approximation \eqref{hjintro} in $L^p$ norms.
	These are done in Sections \ref{sec;rate1} and \ref{sec;rate2};
	\item Stability estimates for Hamilton-Jacobi equations, including numerical schemes of Godunov-type. These are done in Section \ref{sec;num};
	\item Regularizing effects for the Hamilton-Jacobi equation \eqref{hjintroinv} posed on the whole space. These are the subject of Section \ref{sec;regeff};
	\item Liouville-type properties for first- and second-order equations, see Section \ref{sec;lio}.
\end{itemize}

Among the main results, we will prove by duality techniques the following rates of convergence for the vanishing viscosity process of semiconcave solutions to \eqref{hjintro}. First, we obtain the one-side rate of order $\mathcal{O}(\eps)$
\[
\|(u_\eps-u)^+\|_{L^\infty(\R^n)}\leq C\eps,
\]
see  Theorem \ref{ev1}. We discuss its sharpness in Example \ref{ex2rate},  and  the  dependence of the constant $C$ on the dimension of the ambient space, see Remark \ref{constants}. The best rate from below on the difference $u_\eps-u$ known in the literature is the classical $\mathcal{O}(\sqrt{\eps})$ bound for Lipschitz solutions. We try to enhance this order, at the expenses of confining the analysis on compact manifolds, by investigating (two-side) rates of convergence in   $L^p$ norms. In particular, we prove in Theorem \ref{rateL1} that semiconcave solutions of \eqref{hjintro} satisfy the estimate
\[
\|u_\eps-u\|_{L^\infty(0,T;L^1(\T^n))}\leq C\eps.
\]
This is crucial to achieve the interpolated bound, cf. Corollary \ref{prate},
\[
\|u_\eps-u\|_{L^\infty(0,T;L^p(\T^n))}\leq C\eps^{\frac12+\frac{1}{2p}},
\]
which appears to our knowledge the best two-side estimate available in the literature for semiconcave solutions. We complement our analysis with $L^p$ convergence of gradients in the vanishing viscosity process, see Remark \ref{rateDu}. Our results improve the known requirements on the data in  the literature, see e.g. Section \ref{sec;secest} and Remark \ref{weaken}. 
One could expect, possibly combining the effect of the diffusion and certain nonlinearities $H$ (for instance, uniformly convex ones), a linear rate in the viscosity for the estimate from below on $u_\eps-u$, but this remains open. \\

Another development is the application of duality methods to address the regularization effect
\[
\|Du_\eps(t)\|_{L^\infty(\R^n)}\leq \frac{1}{C_H t^\frac1\gamma}\left(\mathrm{osc}_{\R^n\times(0,T)}u_\eps\right)^\frac1\gamma.
\]
for Hamiltonians of the form $H=H(p)$ having power-growth $|p|^\gamma$ with $\gamma>1$ (see \eqref{H1} below).
Remarkably, such bounds are sharp, cf. Remark \ref{firstorder}, and imply new Liouville properties for first-order equations, being independent of the viscosity. They also aim to complete some regularity results for equations with nonlocal diffusion initiated in \cite{Silvestre}, as detailed in Section \ref{coercive}.\\

In perspective, the previous results can  lead to new developments in the context of Mean Field Games, local and nonlocal   fully nonlinear equations  and Porous-Medium equations, that will be the matter of future research.
As a further advance, our duality method suggests, by means of \eqref{semicintro}, a connection among second-order properties for first-order Hamilton-Jacobi equations and   integrability conditions on the advection term for continuity equations arising from the Ambrosio-Diperna-Lions theory \cite{Ambrosio,Figalli,LBL}. Indeed, when $D^2u\leq c(t)\in L^1_t$ and the velocity field of \eqref{adjointintro} with $\eps=0$ is $b(x,t)=-D_pH(x,Du)$, as in Mean Field Games theory, we have from \eqref{semicintro}
\[
D^2u\leq c(t)\in L^1_t\implies [\mathrm{div}(b)]^-\leq c_Hc(t)\in L^1_t,\text{ namely }[\mathrm{div}(b)]^-\in L^1_t(L^\infty_x).
\]
This is a classical condition which, along with a suitable growth of the coefficients in the equation,   ensures the validity of stability estimates in Lebesgue spaces for inviscid transport and continuity equations. This observation plays a fundamental role in the proof of the $L^1$ rate of convergence of the vanishing viscosity process. We also outline in the course of the paper several connections with estimates for conservation laws, see e.g. Remarks \ref{conlaws}, \ref{system} and \ref{Nwave}, for which we propose a new proof by means of duality methods. In particular, Remark \ref{system} provides a new rate of  convergence for a special class of hyperbolic systems for which the convergence to the inviscid system was proved in \cite{K67II}.\\
To conclude, we recall that since its introduction by L.C. Evans \cite{Evansadjoint}, the nonlinear adjoint method has been intensively applied  to several different contexts in the field of nonlinear PDEs, mostly of Hamilton-Jacobi type. To mention a few:  weak KAM theory \cite{EvansWeakKAM,TranBook},  infinity Laplacian equations \cite{EvansSmart}, large-time behavior of Hamilton-Jacobi equations and  Mean Field Games \cite{CGMT,CirantPorretta}, uniqueness principles for weak solutions of viscous Hamilton-Jacobi equations \cite{CGpoinc} and, more recently, maximal regularity properties for viscous Hamilton-Jacobi equations in $L^p$ spaces, in connections with a P.L. Lions' conjecture, and  classical regularity of solutions to Mean Field Games \cite{CGpoinc,CGpar,G23,GPV}.

\section{Some preliminary notions and definitions}
In this section, we    introduce various notions of unilateral second-order bounds which arise in the study of Hamilton-Jacobi equations. 

We start with the classical  definition of semiconcavity.
\begin{defn}\label{semic}\hfill
	\begin{itemize}
		\item[(i)] 
		A function $u\in C(Q)$  is said to be semiconcave with constant $C$ if it satisfies
		\[
		u(x+z,t)-2u(x,t)+u(x-z,t)\leq C|z|^2,\ x,z\in\R^n, t\in(0,T).
		\]
		\item[(ii)]	A function $u\in C(Q)$ is said to be semiconcave for positive times if there exists a constant $C$ such that
		\[
		u(x+z,t)-2u(x,t)+u(x-z,t)\leq C\left(1+\frac1t\right)|z|^2,\ x,z\in\R^n, t\in(0,T).
		\]
	\end{itemize}
\end{defn}
Some remarks on the previous definition are in order
\begin{rem}
	Definition \ref{semic}-(i) is equivalent to the existence of a constant $C>0$ such that $u(x,t)-\frac12 C|x|^2$ is concave on $\R^n$, see e.g. Proposition 5.2 of \cite{BardiDragoni} or \cite{CannarsaSinestrari}. Moreover, a result of P.-L. Lions \cite{Lions83} and H. Ishii \cite{IshiiFE} shows that the semiconcavity is also equivalent to the validity of $D^2u\leq CI_n$ for $t>0$ in the viscosity sense.\\

Definitions \ref{semic}-(i) and (ii) are also equivalent to the following properties, respectively:
\begin{itemize}
		\item[(i)] 
		A function $u\in C(Q)$  is  semiconcave if and only if
		\[
		u_{ee}:=D^2ue\cdot e\leq C\text{ in } \mathcal{D}'(Q)\,,\ \forall e\in\R^n,
		\]
		namely when $D^2u\leq CI_n$ in the sense of distributions.
		\item[(ii)]	A function $u\in C(Q)$ is   semiconcave for positive times if and only if there exists a constant $C$ such that
		\[
		u_{ee}\leq C\left(1+\frac1t\right)\text{ in } \mathcal{D}'(Q),\ \forall e\in\R^n,\ t>0.
		\]
	\end{itemize}
	We refer to \cite{BardiDragoni} for other equivalent characterizations of convexity/concavity in terms of  fully nonlinear Hessian PDEs.
\end{rem}
\begin{rem}
	Definition \ref{semic}-(i) will be used if the initial datum of the problem is semiconcave,  see  for instance  Theorem \ref{semiconc1}. Instead, the semiconcavity for positive times (ii) is related to non-semiconcave initial and source term data and it is connected with   mild regularization effects of solutions to Hamilton-Jacobi equations, see   Theorem \ref{semic>0}  and  Section 3.3 of \cite{EvansBook}. \end{rem}
Following \cite{LinTadmor}, we introduce a generalization of the notion of semiconcavity connected with the  $L^p$ properties of solutions to Hamilton-Jacobi equations.
\begin{defn}\label{semicmix}
	 A function $u\in C(Q)$ is said to be $L^q_t(L^p_x)$-semiconcave if there exists a function $k(x,t)$ in  $L^q(0,T;L^p_{\mathrm{loc}}(\R^n))$, for   $1\leq q,p<\infty$,    and  in $L^q(0,T;L^\infty_{\mathrm{loc}}(\R^n))$, for $p=\infty$, such that
	\[
	u_{ee}\leq k\text{ in } \mathcal{D}'(Q),\qquad \forall e\in\R^n,|e|=1.
	\]
\end{defn}

When $p=\infty$ and $q=1$, the above definition reduces to the semiconcave stability in Definition 2.1 of \cite{LinTadmor}. \\
Following   \cite[Remark 3.6]{L82book}, we introduce
a weaker second-order regularity notion for   solutions of Hamilton-Jacobi equations:
\begin{defn}\label{SSHp}
	A function $u\in C(Q)$ is said to be $L^q_t(L^p_x)$-semisuperharmonic ($L^q_t(L^p_x)$-SSH in short) if there exists a function $k(x,t)$ as in Definition \ref{semicmix} such that 
	\[
	\Delta u\leq k(x,t)\text{ in } \mathcal{D}'(Q).
	\]
\end{defn}
	A result of H. Ishii (see \cite{IshiiFE}) shows that,  for $p=q=\infty$, Definition \ref{SSHp} is equivalent 
	to $u$ being a viscosity supersolution of $C-\Delta u=0$ in $Q$.\\
We now introduce the notions of solution for the Hamilton-Jacobi equations. In all the paper, we assume that $H$, $f$ and $u_0$ are at least continuous. For $\eps>0$, we consider classical solutions to \eqref{hjintro} (existence and uniqueness results can be found in \cite{L82book}).
\begin{defn}
	A classical solution to \eqref{hjintro} is a function $u\in C^{2,1}_{x,t}(Q)$ solving the problem in pointwise sense.
\end{defn} 
When $H(x,Du)$ is bounded, $C^{2,1}$ regularity can be achieved by classical maximal regularity for heat equations, see \cite{CGpar}.\\
To study some properties of solutions in the limit $\eps=0$, we  will consider two notions of solutions: viscosity solutions and generalized solutions.  We start with that of viscosity solution, cf. \cite{CEL,CL83tams}.
\begin{defn}\label{visco}
	A continuous function  $u:Q\to\R$ is said to be a viscosity subsolution (respectively, supersolution) of \eqref{hjintroinv}
	if for any $(x_0,t_0)\in Q$ and for any   $\phi\in C^1(Q)$ such that $u-\phi$
	has a local maximum (respectively, minimum) point at $(x_0,t_0)$, then we have
	\begin{align*} 
		&\partial_t\phi(x_0,t_0)+  H(x_0, D\phi(x_0,t_0))\le f(x_0,t_0)\\
	\Big(\text{respectively,}\quad &\partial_t\phi(x_0,t_0)+  H(x_0, D\phi(x_0,t_0))\ge  f(x_0,t_0)  \Big)
	\end{align*}
and $u(x)\le u_0(x)$ (respectively, $u(x)\ge u_0(x)$) in $\R^n$.
A continuous function $u:Q\to \R$ is said to be a viscosity solution of \eqref{hjintroinv}
if it is a viscosity subsolution and supersolution. 
\end{defn}
The second one is that of generalized solution in the sense of S.N. Kruzhkov \cite{K67II}.
\begin{defn}\label{def;weakae}
	A Lipschitz continuous function $u:Q\to\R$ is said to be a generalized solution of \eqref{hjintroinv}
	provided that
	\begin{itemize}
		\item[(i)] $u(x,0)=u_0(x)$ for $x\in\R^n$, with $u_0$ Lipschitz;
		\item[(ii)] the equation $\partial_t u+H(x,Du)=f(x,t)$ holds a.e. on $Q$;
		\item[(iii)] $u$ is semiconcave for positive times on $Q$.
	\end{itemize}
\end{defn}
\begin{rem}
	Both the definitions of solutions for \eqref{hjintroinv} satisfy  existence and  uniqueness results. As for Definition \ref{visco} we refer to \cite{BCD,CEL} for general Hamiltonians $H$, while corresponding results for generalized solutions are due to S.N. Kruzhkov \cite{Kruzhkov}, but only in the case  $H\in C^2$ and convex in $p$. \\
	In P.-L. Lions \cite[Theorems 3.1 and 10.1]{L82book},  assumption (iii) in Definition \ref{def;weakae} was relaxed to $L^\infty_{x,t}$-SSH, obtaining existence and  uniqueness results for  SSH solutions to Hamilton-Jacobi equations. When $u$ is $L^1_t(L^p_x)$-SSH and solves a stationary Hamilton-Jacobi equation with $H$ convex, P.-L. Lions proved in \cite{L82book} a stability estimate in $L^\infty$. This result is however conditional to the validity of unilateral bounds in $L^p$ of $\Delta u$ and cannot be reached, as the author in \cite{L82book} explains, using the sole maximum principle. We provide a way to derive these bounds in Section \ref{sec;semiLp}.
\end{rem}
These two notions of solutions for \eqref{hjintroinv} are intimately connected as shown in the next proposition, cf. \cite{L82book} and Proposition III.3 in \cite{CL83tams}.
\begin{prop}
	Let $H$ be convex in $p$. If $u$ is a weak a.e. subsolution of \eqref{hjintroinv}  (i.e., in the sense of Definition \ref{def;weakae} but without assuming (iii)), then $u$ is a viscosity subsolution of \eqref{hjintroinv}. If $u$ is a generalized supersolution of \eqref{hjintroinv}, then $u$ is a viscosity supersolution of \eqref{hjintroinv}. Hence, if $u$ is a generalized solution to \eqref{hjintroinv},   then $u$ is a viscosity solution to the same problem.
\end{prop}

\section{Second-order estimates for Hamilton-Jacobi equations}\label{sec;secest}
We start by reviewing the literature on (one-side) second-order bounds  for solution to Hamilton-Jacobi equations. The following list summarizes the main approaches to derive semiconcavity estimates and convexity preserving properties for  semilinear equations of Hamilton-Jacobi type: 
\begin{itemize}
\item[(a)] The first one consists in approximating the first-order problem adding a viscosity term $\eps\Delta u_\eps$. Using the maximum principle   for elliptic and parabolic equations to derive a semiconcavity estimate independent of $\eps>0$, one brings to the limit as $\eps\to 0$ the properties of the solution of the second-order equation to the first-order one;
\item[(b)] The second one is based on control theoretic techniques working at the level of representation formulas, an example being the Hopf-Lax formula.
\item[(c)] Another method exploits the doubling of variable argument in the viscosity solutions theory, which avoids the differentiation of the equation. Indeed, the low regularity of viscosity solutions to first-order problems prevents from the formal differentiation of the equation as in the viscosity regularization described in (a). 
\item[(d)] Finally, spatial convexity preserving properties can be proved by the analysis of the convex envelope if the comparison principle holds.
\end{itemize}
All the above methods usually require convexity-type assumptions of the Hamiltonian. The parabolic regularization to prove semiconcavity (or SSH) properties was first used by S. N. Kruzhkov \cite{K66I,K67II}, see also P.-L. Lions \cite{L82book}. The method in (c) is due to H. Ishii and P.-L. Lions \cite{IL} for second-order fully nonlinear equations, see also \cite{Barles,Calder,CannarsaSinestrari} for a proof for first-order equations and \cite{Giga} for convexity preserving properties. These works have their roots in \cite{Korevaar}, where a concavity maximum principle was established. We do not focus on the approach in (b), referring to \cite{CannarsaSinestrari} for further details, mentioning, however, that it allows weaker differentiability properties on $H$ and, sometimes, it encompasses also some non-convex cases, cf.  \cite[Theorem 5.3.9]{CannarsaSinestrari}. The approach in (d), based on the properties of the convex envelope of solutions to elliptic equations, was developed in \cite{ALL} and applies even to fully nonlinear equations. A survey on the previous approaches can be found in Section 5.3 of \cite{CannarsaSinestrari}, while other  recent results based on different  techniques can be found in \cite{Liu}.\\
The recent introduction of the nonlinear adjoint method by L.C. Evans \cite{Evansadjoint} provides an alternative to the use of the maximum principle in the method   (a) and allows to get  most of the known results  by means of stability properties of the (dual) Fokker-Planck equations, see also \cite{GPV,TranBook,Tran2011} and the more recent   \cite{CGsima}. \\
Exploting the adjoint method, we  obtain second-order bounds for solutions to Hamilton-Jacobi equations under suitable assumptions on $H$ that do not necessarily involve convexity or strict convexity. This will be done first assuming semiconcavity properties of  the data  and then removing this assumption, showing thus a regularizing effect. Moreover,  via a refinement of the   method, we will   prove some new one-side second-order bounds in $L^p$ spaces inspired by \cite{LinTadmor,L82book} that provide new results for the rate of convergence of the vanishing viscosity method. This will also complement the results in \cite{L82book} about uniqueness and stability for convex Hamilton-Jacobi equations.

We consider  the viscous Hamilton-Jacobi equation \eqref{hjintro} and suppose that $H$ is $C^2(\R^n\times\R^n)$, $H(x,p)\geq H(x,0)=0$ and there exist constants $\gamma > 1$ and $C_{H,i}>0$, $\tilde{C}_{H,i}\geq0$ such that
\begin{align}
\tag{H1}\label{H1} & D_pH(x,p)\cdot p-H(x,p)\geq C_{H,1}|p|^{\gamma}-\tilde{C}_{H,1}\ , \\
\tag{H2}\label{H2} & |D_{xx}^2H(x,p)|\leq C_{H,2}|p|^{\gamma}+\tilde{C}_{H,2} \ , \\
\tag{H3}\label{H3} & |D_{px}^2H(x,p)|\leq C_{H,3}|p|^{\gamma-1}+\tilde{C}_{H,3} \ , \\
\tag{H4}\label{H4} & D_{pp}^2H(x,p)\xi\cdot \xi\geq C_{H,4}|p|^{\gamma-2}|\xi|^2-\tilde{C}_{H,4}. 
\end{align}
Note that \eqref{H1} and \eqref{H4} are convexity-type assumptions, but weaker than convexity. For instance, the Hamiltonian
\[
H_1(p)=\frac{(|p|^2-1)^2-1}{|p|^2+1}
\]
is nonconvex and satisfies the assumption \eqref{H1} with $\gamma=2$, $C_{H,1}=1$ and $\widetilde{C}_{H,1}=-6$. It can be written as
\[
H_1(p)=|p|^2+z(p),\ z(p)=-3+\frac{3}{1+|p|^2}\leq0.
\]
Another nonconvex example satisfying both \eqref{H1} and \eqref{H4} is
\[
H_2(p)=(|p|^2-1)^2-2.
\]
In fact, we have
\[
D_pH_2(p)\cdot p-H_2(p)=4|p|^2(|p|^2-1)-(|p|^2-1)+2\geq 3(|p|^2-1)^2+2.
\]
Using now that $(a-b)^2\geq \frac{a^2}{2}-b^2$ for all $a,b\in\R$ we conclude
\[
D_pH_2(p)\cdot p-H_2(p)\geq 3|p|^4-1.
\]
It is easy to check \eqref{H4}, at least when $n=1$. \\
A strictly (but not uniformly) convex function satisfying both \eqref{H1} and \eqref{H4} with $\gamma\in(1,2)$ is
\[
H_3(p)=(\eta+|p|^2)^{\frac{\gamma}{2}},\ \eta>0.
\]
Indeed, we have
\[
D_pH_3(p)\cdot p-H_3(p)=\left(\gamma\left(\eta + |p|^2\right)^{\frac{\gamma-2}{2}}|p|^2-(\eta+|p|^2)^{\frac{\gamma}{2}}\right).
\]
and \eqref{H1} holds with $C_{H,1}=\gamma-1$ and $\tilde{C}_{H,1}=\eta^\frac{\gamma}{2}$. Indeed,
\[
\gamma\left(\eta + |p|^2\right)^{\frac{\gamma-2}{2}}|p|^2-(\eta+|p|^2)^{\frac{\gamma}{2}}=-\frac{\eta}{\left(\eta + |p|^2\right)^{\frac{2-\gamma}{2}}}\geq -\eta^{\frac{\gamma}{2}}.
\]
To check \eqref{H4} (more precisely, its regularized version $D_{pp}^2H(p)\xi\cdot \xi\geq C_{H,4}(\eta+|p|^2)^\frac{{\gamma-2}}{2}|\xi|^2$), note that
\[
D^2_{p_ip_j}H_3(p)=\gamma\left((\gamma-2)\left(\eta + |p|^2\right)^{\frac{\gamma-4}{2}}p_ip_j+\delta_{ij}(\eta+|p|^2)^{\frac{\gamma}{2}-1}\right).
\]
Then, using that $\gamma\in(1,2)$ and the Cauchy-Schwarz inequality we get
\begin{align*}
D^2_{p_ip_j}H_3(p)\xi_i\xi_j&\geq \gamma\left(|\xi|^2(\eta+|p|^2)^{\frac{\gamma}{2}-1}-(2-\gamma)\left(\eta + |p|^2\right)^{\frac{\gamma-4}{2}}(p\cdot \xi)^2\right)\\
&\geq \gamma\left(|\xi|^2(\eta+|p|^2)^{\frac{\gamma}{2}-1}-(2-\gamma)\left(\eta + |p|^2\right)^{\frac{\gamma-4}{2}}|p|^2|\xi|^2\right)\\
&=\gamma|\xi|^2(\eta+|p|^2)^{\frac{\gamma}{2}-2}\left(\eta+|p|^2-(2-\gamma)|p|^2\right)\\
&=\gamma|\xi|^2(\eta+|p|^2)^{\frac{\gamma}{2}-1}\frac{(\gamma-1)|p|^2+\eta}{\eta+|p|^2}\\
&\geq \gamma(\gamma-1)|\xi|^2(\eta+|p|^2)^{\frac{\gamma}{2}-1}.
\end{align*}

Some remarks on the hypotheses \eqref{H1} and \eqref{H4} are in order. Regarding semiconcavity estimates, these are typically obtained when $H$ is uniformly convex. However, in this context, we derive second-order bounds by relaxing this traditional condition, at the cost of imposing the coercivity assumption \eqref{H1}, highlighting that the estimate is independent of the viscosity. To our knowledge, \eqref{H1} has already been employed to derive one-side second-order bounds for Hamilton-Jacobi equations with mixed local and nonlocal diffusion, under stronger assumptions on the data (cf. \cite{CGsima}). As a byproduct, this analysis allows us to relax the known assumptions on $H$ to establish the rate of convergence for the vanishing viscosity process, which will be discussed in detail in Section \ref{sec;rate1} and Section \ref{sec;rate2}.\\
	As for first-order estimates, mild coercivity conditions on $H$ are required to obtain gradient bounds for these nonlinear equations; see, for example, \cite{L82book}. To our knowledge, Hamiltonians $H$ satisfying conditions similar to \eqref{H1} have appeared in the theory of viscosity solutions. See \cite[hypothesis (H8) p. 1319]{BSouga} as a general reference, where the authors use maximum principle methods (via the doubling of variables technique) to prove a Lipschitz-preserving estimate. The work in \cite{Lions85aa} proved Lipschitz regularizing effects under the condition
		\[
		D_p H(p) \cdot p - H(p) \to +\infty \quad \text{as} \quad |p| \to +\infty,
		\]
		obtaining a result similar to Theorem \ref{regLip} through the Bernstein method combined with the maximum principle. The coercivity condition \eqref{H1}, for $H = H(x, p)$ depending smoothly on $x$, was recently used in \cite{CGpoinc} to study Lipschitz smoothing effects, even in the presence of $L^p$ right-hand sides, using the nonlinear adjoint technique. However, our result in Theorem \ref{regLip} is more aligned with the analysis in \cite{Lions85aa} than in \cite{CGpoinc}, as the regularization mechanism arises from the nonlinearity, rather than from the diffusion. In fact, the estimate in Theorem \ref{regLip} is independent of the viscosity.

\subsection{Second-order one-side bounds under semiconcavity assumptions}
We start with the following example, which motivates the main result of the section on the conservation of semiconcavity properties for Hamilton-Jacobi equations
\begin{ex}[Example (ii) in Section 3.3 of \cite{EvansBook}] Consider  the initial-value problem
\[
\begin{cases}
\partial_t u+\frac12|Du|^2=0&\text{ in }\R^n\times(0,\infty),\\
u(x,0)=-|x|&\text{ in }\R^n.
\end{cases}
\]
 The unique viscosity solution of the problem can be find  by means of the Hopf-Lax formula
\[u(x,t)= \inf_{y\in\R^n}\left\{u_0(y)+tL\left(\frac{x-y}{t}\right)\right\},\]
where $L=H^*$ is the Fenchel conjugate of $H$, which in this case gives
\[
u(x,t)=
-|x|-\frac{t}{2},\ t\geq0.
\]
The initial condition is semiconcave,  see   \cite[Example 2.2.5]{CannarsaSinestrari}, and the above explicit formula shows that the solution preserves the semiconcavity for positive times.
\end{ex}
We start by proving the preservation of semiconcavity properties from the data. From now on, we will exploit integral methods to get zero, first- and second-order \textit{a priori} estimates, so we implicitly assume that $u$ and its derivatives belong to a suitable energy class when deriving the next estimates by duality. This in particular allows to justify the variational formulations we use throughout the proofs.\\
From now on, throughout the proofs we use $ u$ to denote $u_\eps$, but we will take care of the dependence of the constants on $\eps$. We premise the following
\begin{lemma}\label{cross}
Let $u_\eps$ be a solution to \eqref{hjintro} with $H$ satisfying \eqref{H1}. There exists a constant $K$ depending on $\|u_\varepsilon\|_{L^\infty_{x,t}}$, $C_{H,1}$, $\widetilde{C}_{H,1}$ and independent of $\eps$ such that
\[
\iint_{Q_\tau}|Du_\varepsilon|^\gamma\rho\,dxdt\leq K,\ \tau\in(0,T],
\]
where $\rho$ is the solution to the linear advection-diffusion equation 
\begin{equation}\label{FK_test}
\begin{cases}
-\partial_{t} \rho - \eps \Delta \rho - \mathrm{div} (D_pH(x, Du) \rho) = 0, & \text{ in }\R^n\times(0,\tau),
\\
\rho(x, \tau) = \rho_{\tau}(x), & \text{ in } \R^n.
\end{cases}
\end{equation}
with $\rho_{\tau} \in C^{\infty}_{c}(\R^n)$, $\rho_{\tau} \geq 0$ and $\|\rho_{\tau}\|_{L^1(\R^n)} = 1$. 
\end{lemma}
\begin{proof}
To show the  estimate, we test the solution $\rho$ to \eqref{FK_test} against the solution of the Hamilton-Jacobi equation, and use the solution $u$ as a test function for the transport equation \eqref{FK_test}. We thus have 
\begin{multline*}
\int_{\R^n} u(x, \tau) \rho_{\tau}(x)\,dx - \int_{\R^n} u(x, 0) \rho_{0}(x)\,dx -\int_{0}^{\tau}\int_{\R^n}f(x,t)\rho(x,t)\,dxdt
\\= \int_{0}^{\tau} \int_{\R^n} (D_pH(x, Du(x, \tau)) \cdot Du(x, \tau) - H(x, Du(x, \tau))\rho(x, \tau)\;dxd\tau.
\end{multline*}
and, by \eqref{H1}, we deduce 
\begin{multline*}
\int_{0}^{\tau} \int_{\R^n} (D_pH(x, Du(x, \tau)) \cdot Du(x, \tau) - H(x, Du(x, \tau))\rho(x, \tau)\;dxd\tau \\
\geq C_{H,1}\int_{0}^{\tau} \int_{\R^n} |Du|^{\gamma}\rho\;dxdt-\widetilde{C}_{H,1}\int_0^\tau\int_{\R^n}\rho\,dxdt.
\end{multline*}
We conclude  that
\begin{multline}\label{Dugamma}
	\int_{0}^{\tau} \int_{\R^n} |Du|^{\gamma}\rho\;dxdt \leq \frac{1}{C_{H,1}}\left[\int_{\R^n} u(x, \tau) \rho_{\tau}(x)\,dx - \int_{\R^n} u(x, 0) \rho_{0}(x)\,dx\right.\\
	\left.-\int_{0}^{\tau}\int_{\R^n}f(x,t)\rho(x,t)\,dxdt-\widetilde{C}_{H,1}\int_0^\tau\int_{\R^n}\rho\,dxdt\right]=:K.
\end{multline}
Note that $K<\infty$ since $\rho\in L^\infty([0,\tau];L^1(\R^n))$ (or even by using merely that $\int_{\R^n}\rho(t)\,dx\leq1$) by Lemma \ref{well}.
\end{proof}
We are now ready to prove the main result of the section.
\begin{thm}\label{semiconc1}
Assume that $H$ satisfies \eqref{H1}-\eqref{H4} 
 with $u_0:\R^n\to\R$ bounded and semiconcave, i.e. $D^2u_0\leq c_0I_n$, and $f\in C(Q)$, bounded and $L^1_t(L^\infty_x)$-semiconcave, i.e. $D^2f\leq c_f(t)I_n$ with $c_f\in L^1(0,T)$. Then, the classical solution to \eqref{hjintro}  is semiconcave, i.e. 
\begin{equation}\label{sem}
D^2u_\eps\leq CI_n,
\end{equation}
where $C$ depends on the constants $C_{H,i}$, $c_0$, $\int_0^Tc_f(t)\,dt$, $\|u_\eps\|_{L^\infty_{x,t}}$ and $\|f\|_{L^\infty_{x,t}}$.  Moreover, if $H=H(p)$,  we have the explicit bound
\begin{equation}\label{sem2}
D^2u_\eps(\tau)\xi\cdot \xi\leq c_0+\int_0^\tau c_f(t)\,dt+\widetilde{C}_{H,4}\tau, \quad \forall \xi\in\R^n,|\xi|=1, \tau \in (0,T).
\end{equation}
\end{thm}
\begin{rem}
A similar result continues to hold if the semiconcavity assumptions on the data are replaced by weaker SSH bounds, i.e. if $\Delta u_0\leq c_0$ and $\Delta f\leq c_f(t)$, $c_f\in L^1(0,T)$. These assumptions imply estimates on $\Delta u_\eps$ instead of the full Hessian (in the sense of measures) using the same method, providing a different proof than those obtained in \cite{L82book} exploiting  the maximum principle.
\end{rem}

\proof
Let $\eta \in C^{\infty}_c(\R^n)$, let $\eta_{\delta}(x) = \frac{1}{\delta^n}\eta(\frac{x}{\delta})$  and set $f_{\delta}(x, t)= \eta_{\delta} \star f(x, t)$. Then,  $f_{\delta}$   uniformly converges to $f$ as $\delta \downarrow 0$ and $\{f_{\delta}\}_{\delta >0}$ is a family of semiconcave functions with the same modulus $c_{f}(t)$, indeed given $\xi \in \R^n$ with $|\xi| = 1$
\begin{equation*}
D^2 f_{\delta}(x, t) \xi \cdot \xi = \int_{\R^n} \eta_{\delta}(y)D^2f(x-y)\xi \cdot \xi \;dy \leq \int_{\R^n} \eta_{\delta}(y)c_f(t)\;dy = c_f(t).
\end{equation*}
Let us consider the classical solution $u^{\delta}_{\eps}$ to the regularized Hamilton-Jacobi equation
\begin{equation}\label{HJ_reg}
\begin{cases}
\partial_t u_\eps^{\delta}-\eps\Delta u^{\delta}_\eps+H(x,Du^{\delta}_\eps)=f_{\delta}(x,t)&\text{ in }Q,\\
u^{\delta}_\eps(x,0)=u_0(x)&\text{ in }\R^n.
\end{cases}
\end{equation}
Observe that, since $f_{\delta}(\cdot, t)$ belongs to $C^{2}(\R^n)$ then we have $u^{\delta}_{\eps}(\cdot, t) \in C^4(\R^n)$.  One also needs to further regularize the initial datum as done at p.19 of \cite{CGpar}, but we avoid this technical step as introduces an additional regularization parameter.
For simplicity of notation we drop the indices $\eps, \delta$ in $u$ and we set $u=u^{\delta}_{\eps}$. Differentiating twice the equation, we get for $v=u_{ee}$
\begin{multline}\label{xixi}
\partial_{t}v-\eps\Delta v+D^2_{pp}H(x, Du)Du_e\cdot Du_e +D_pH(x, Du)\cdot Dv
\\
+ 2 D_{px}^{2} H(x, Du)Du_{e} + D^{2}_{xx}H(x, Du) =(f_{\delta})_{ee}.
\end{multline}
Multiplying \eqref{xixi} by $\rho$ and integrating over $\R^n \times [0, \tau]$, we get 
\begin{multline*}
\int_{\R^n} v(x, \tau) \rho_{\tau}(x)\,dx + \int_{0}^{\tau} \int_{\R^n} Du_{e} \cdot D^{2}_{pp}H(x,Du) Du_{e} \rho\, dxdt = 
\\
- \int_{0}^{\tau} \int_{\R^n} \left(2D_{px}^{2} H(x, Du) Du_{e} + D^{2}_{xx}H(x, Du) \right)\rho\,dxdt\  + \int_{\R^n} v(x, 0)\rho_0(x)\,dxdt +\int_{0}^{\tau}\int_{\R^n}(f_{\delta})_{ee}\rho\,dxdt. 
\end{multline*}

On the  one hand, by \eqref{H4} we have 
\begin{equation*}
\int_{0}^{\tau} \int_{\R^n}  D^2_{pp}H(x, Du)Du_{e} \cdot Du_{e}\rho\,dxdt \geq C_{H,4} \int_{0}^{\tau} \int_{\R^n} |Du|^{\gamma-2} |Du_{e}|^{2}\rho\, dxdt - \widetilde{C}_{H,4} \int_{0}^{\tau} \int_{\R^n} \rho\,dxdt
\end{equation*}
Moreover, by \eqref{H2} and \eqref{H3},  we obtain 
\begin{multline*}
\int_{\R^n} v(x, \tau) \rho_{\tau}(x)\,dx +C_{H,4} \int_{0}^{\tau} \int_{\R^n} |Du|^{\gamma-2} |Du_{e}|^{2}\rho\, dxdt - \widetilde{C}_{H,4} \int_{0}^{\tau} \int_{\R^n} \rho\,dxdt
\\
\leq  C_{H, 2} \int_{0}^{\tau}\int_{\R^n} |Du|^{\gamma}\rho\;dxdt\ + C_{H, 3} \int_{0}^{\tau} \int_{\R^n} |Du|^{\gamma-1} |Du_{e}|\rho\;dxdt\ + (\widetilde{C}_{H,2} + \widetilde{C}_{H,3}) \int_{0}^{\tau} \int_{\R^n} \rho\;dxdt\
\\
+  \int_{\R^n} v(x, 0)\rho_0(x)\,dxdt +\int_{0}^{\tau}\int_{\R^n}(f_{\delta})_{ee}\rho\,dxdt. 
\end{multline*}
By the weighted Young's inequality we get for all $\sigma>0$
\begin{equation*}
C_{H,3}\int_{0}^{\tau} \int_{\R^n} |Du|^{\gamma-1} |Du_{e}|\rho\;dxdt \leq \frac{\sigma^2 C_{H,3}}{2} \int_{0}^{\tau} \int_{\R^n} |Du|^{\gamma-2} |Du_{e}|^{2}\rho\;dxdt\ + \frac{C_{H,3}}{\sigma^2} \int_{0}^{\tau} \int_{\R^n} |Du|^{\gamma}\rho\;dxdt.
\end{equation*}
We choose $\frac{\sigma^2 C_{H,3}}{2}=C_{H,4}$, and conclude that
\begin{equation}\label{stima-1}
	\begin{split}
\int_{\R^n} v(x, \tau) \rho_{\tau}(x)\,dx &\leq \int_{\R^n} v(x, 0)\rho_0(x)\,dxdt +\int_{0}^{\tau}\int_{\R^n}(f_{\delta})_{\xi\xi}\rho\,dxdt 
\\
&+ \left(\frac{C_{H,3}^2}{4C_{H,4}} + C_{H,2}  \right) \int_{0}^{\tau} \int_{\R^n} |Du|^{\gamma}\rho\;dxdt.
\end{split}
\end{equation}
By Lemma \ref{cross}, we have
\begin{equation*}
\int_{0}^{\tau} \int_{\R^n} |Du|^{\gamma}\rho\;dxdt\leq K,
\end{equation*} 
and in view of Lemma \ref{well} 
\begin{multline*}
\int_{\R^n} u(x, \tau) \rho_{\tau}(x)\,dx - \int_{\R^n} u(x, 0) \rho_{0}(x)\,dx-\int_{0}^{\tau}\int_{\R^n}f(x,t)\rho(x,t)\,dxdt\\
+\widetilde{C}_{H,1}\int_{0}^{\tau}\int_{\R^n}\rho\,dxdt\leq 2\|u\|_\infty+(\|f\|_{\infty}+\widetilde{C}_{H,1})\tau.
\end{multline*}
Therefore, passing to the supremum over $\rho_{\tau}$ in \eqref{stima-1}, we deduce for $u=u^{\delta}_{\eps}$
\begin{equation*}
D^2u^{\delta}_{\eps}e\cdot e\leq c_0+\int_0^T\int_{\R^n}(f_{\delta})_{ee}\rho\,dxdt +\left(\frac{C_{H,3}^2}{4C_{H,4}} + C_{H,2}  \right)K\leq c_0+\int_0^Tc_f(t)\,dt+\left(\frac{C_{H,3}^2}{4C_{H,4}} + C_{H,2}  \right)K,
\end{equation*}
where we used that the  $\{f_{\delta}\}_{\delta > 0}$ preserves the modulus of semiconcavity $c_f(t)$. Finally, passing to the limit as $\delta \downarrow 0$ we get \eqref{sem}.  Piecing together all the estimates we finally conclude 
\begin{multline*}
\int_{\R^n}v(x, \tau) \rho_{\tau}(x)\,dx\leq c_0+\int_0^Tc_f(t)\,dt+\frac{1}{C_{H,1}}\left(\frac{C_{H,3}^2}{4C_{H,4}}+C_{H,2}\right)[2\|u\|_\infty+(\|f\|_{\infty}+\widetilde{C}_{H,1})\tau]\\
+(\widetilde{C}_{H,2} + \widetilde{C}_{H,3}+\widetilde{C}_{H,4})\tau. \hfill\square
\end{multline*}

\begin{rem}
The dependence of the constant  on $\|u_\eps\|_{L^\infty}$ in \eqref{sem} can be removed, and the semiconcavity constant can be bounded only in terms of $\|u_0\|_{L^\infty}$. To show this, one can argue by duality as in Proposition 3.7 of \cite{CGpoinc} using that the solution of the adjoint problem belongs to $L^\infty([0,\tau];L^1(\R^n))$, the only difference with \cite{CGpoinc} being that $f\in L^\infty(Q)$ instead of $f\in L^q(Q)$.
\end{rem}
\begin{rem}\label{Hconvex}
When $H=H(p)$ is  convex, the equation satisfied by $u_{ee}$ reads as
\[
\partial_{t}v-\eps\Delta v+D^2_{pp}H(Du)Du_e\cdot Du_e + D_pH(Du)\cdot Dv=f_{ee}.
\]
By duality one gets the semiconcavity estimate
\[
u_{ee}(x,t)\leq u_{ee}(x,0)+\int_0^T\|(f_{ee}(x))^+\|_{L^\infty(\R^n)}\,dt.
\]
\end{rem}

We conclude with an application of Lemma \ref{cross} to prove the conservation of Lipschitz regularity for \eqref{hjintro}. This is already known by means of the maximum principle, cf. \cite{Barles}, using slightly different hypotheses on $H$. 
\begin{lemma}\label{grad}
Let $H\in W^{1,\infty}_{\mathrm{loc}}(\R^n)$ be satisfying \eqref{H1}. Then any solution $u_\eps$ to \eqref{hjintro} with $u_0\in W^{1,\infty}(\R^n)$ satisfies
\[
\|Du_\eps\|_{L^\infty(Q)}\leq \|Du_0\|_{L^\infty(\R^n)}+\|Df\|_{L^\infty(Q)}T.
\]
In particular, the estimate does not depend on $\eps$.
\end{lemma}
\begin{proof}
We proceed by the adjoint method. After a regularization argument we may assume that $u_\eps$ is sufficiently smooth to perform a differentiation procedure. We set $v=u_e$, $e\in\R^n$, $|e|=1$, to find
\[
\partial_t v-\eps \Delta v+D_pH(Du)\cdot Dv=f_{e}
\]
By duality, using $\rho$ solving \eqref{FK_test} we get
\[
\int_{\R^n}v(\tau)\rho_\tau(x)\,dx\leq \int_{\R^n}v(0)\rho(0)\,dx+\int_0^\tau\int_{\R^n}f_e\rho\,dxdt
\]
We now use that $\|\rho(t)\|_{L^1}=1$ by Lemma \ref{well} and Lemma \ref{cross} to conclude the estimate, recalling that $u_0\in W^{1,\infty}(\R^n)$.
\end{proof}

\begin{rem}
Note that the semiconcavity estimates \eqref{sem} and \eqref{sem2} are independent of the viscosity parameter $\eps$. So, appealing to  \cite[Theorem 3.3.3]{CannarsaSinestrari} and using the global uniform bound of $Du_{\eps}$ (see, for instance, \cite{L82book} or Lemma \ref{grad}) we deduce that the vanishing viscosity limit $u$ is a semiconcave solution to \eqref{hjintroinv} with the same modulus of semiconcavity $C$ as in \eqref{sem}. 
\end{rem}

\begin{rem}
The obtainment of Lipschitz estimates typically requires to impose some mild coercivity assumptions on $H$ with respect to $Du$. This is done for instance in \cite{Barles,L82book} in the theory of viscosity or generalized solutions. Here what we really need is the conservation of mass for the dual problem solved by $\rho$. In the periodic setting this property is automatically satisfied by using the test function identically equal to $1$, and a locally Lipschitz $H$ is enough to run the argument. The whole space $\R^n$ requires more care, and some additional assumptions,   as for instance \eqref{H1}, than the sole local Lipschitz continuity are needed, cf. Lemma \ref{well}.
\end{rem}

\subsection{Second-order regularizing effects for equations with Lipschitz data}
In this section we focus on Hamiltonians depending only on $p$ without source terms in the equation. We show, on the line of \cite{EvansBook} or Proposition 2.2.6 in \cite{CannarsaSinestrari}, that in this case solutions to Hamilton-Jacobi equations satisfy a mild regularization effect even though the initial datum is not semiconcave, provided that the Hamiltonian satisfies convexity-type hypotheses. The next is an explicit example of such a phenomenon and motivates   Theorem \ref{semic>0}. 
\begin{ex}[Example (i) in Section 3.3 of \cite{EvansBook}]\label{ex2} The initial-value problem
\[
\begin{cases}
\partial_t u+\frac12|Du|^2=0&\text{ in }\R^n\times(0,\infty).\\
u(x,0)=|x|&\text{ in }\R^n.
\end{cases}
\]
admits the viscosity solution given by the Hopf-Lax formula
\[
u(x,t)=\min_{y\in\R^n}\left\{\frac{|x-y|^2}{2t}+|y|\right\}.
\]
In particular, one has 
\[
u(x,t)=\begin{cases}
|x|-\frac{t}{2}&\text{ if }|x|\geq t\\
\frac{|x|^2}{2t}&\text{ if }|x|\leq t.
\end{cases}
\]
 The initial condition is not semiconcave, see  \cite[Example 3.3.9]{CannarsaSinestrari}, but the solution becomes semiconcave for positive times. 
\end{ex}
The next  result prove a semiconcavity result for positive times, weakening the requirement of uniform convexity in \cite{EvansBook,Evansadjoint,CannarsaSinestrari}. A similar result was obtained by S. Kruzhkov \cite{K67II} and by W. Fleming \cite{FlemingJDE} using different methods under an assumption similar to \eqref{H4}. Moreover, it provides a first step towards a Lipschitz regularization effect for first-order Hamilton-Jacobi equations  which will be discussed in Section \ref{sec;regeff}. 
\begin{thm}\label{semic>0}
Assume that $H$ satisfies \eqref{H4} and $f\equiv0$ (no further hypotheses are assumed on $u_0$). Then, any  solution to \eqref{hjintro} with $Du\in L^\infty_{x,t}$ satisfies 
\[
(u_\eps)_{ee}\leq\frac{C_1}{t}\|Du_\eps\|_\infty^{2-\gamma}+C_2t\text{ when }\gamma\leq 2,
\]
where $C_1,C_2$ depend on $C_{H,4},\widetilde{C}_{H,4}$, and do not depend on $\eps$. Moreover    $C_2=0$ if $\widetilde{C}_{H,4}=0$.
\end{thm}
\begin{proof}
We follow   of an idea introduced by L.C. Evans in Theorem 4.2 of \cite{Evansadjoint}. We differentiate the equation twice with respect to an arbitrary unitary direction $e\in\R^n$ to find for $w=u_{ee}$
\begin{equation}
\partial_t w-\eps\Delta w+D^2_{pp}H(Du)Du_e\cdot Du_e+D_pH(Du)\cdot Dw=0\text{ in }Q.
\end{equation}
Assume for the moment a regularized version of \eqref{H4}, i.e.,
\[
D^2_{pp}H(p)\xi\cdot\xi\geq C_{H,4}(\delta+|p|^2)^{\frac{\gamma-2}{2}}|\xi|^2-\widetilde{C}_{H,4}.
\]
The result will follow by letting $\delta$ to 0. Taking a smooth function $\chi:[0,T]\to\R$ to be chosen later  and setting $z=\chi w$ we have
\[
\partial_t z-\eps\Delta z+\chi D^2_{pp}H(Du)D u_e\cdot D u_e+D_pH(Du)\cdot Dz=\chi'(t)u_{ee}.
\]
We consider the adjoint problem
\[
\begin{cases}
-\partial_t \rho-\eps\Delta \rho-\mathrm{div}(D_pH(Du)\rho)=0&\text{ in }\R^n\times(0,\tau)\\
\rho(x,\tau)=\rho_\tau(x)&\text{ in }\R^n
\end{cases}
\]
and we choose $\chi(t)=t^2$ on $[0,\tau]$, observing that this implies $\int_{\R^n}z(0)\rho(0)\,dx=0$. This choice is crucial to shift the time horizon away from the initial time $t=0$, and avoids to require second-order properties on the initial datum. We test the equation of $z$ by $\rho$ to get
\[
\int_{\R^n}z(\tau)\rho_\tau(x)\,dx+\iint_{Q_\tau}t^2 D^2_{pp}H(Du)Du_e\cdot Du_e\rho\,dxdt=\iint_{Q_\tau}2tu_{ee}\rho\,dxdt.
\]
On the one hand, we have
\[
\iint_{Q_\tau}t^2 D^2_{pp}H(Du)Du_e\cdot Du_e\rho\,dxdt\geq C_{H,4}\iint_{Q_\tau}t^2 |D^2u|^2(\delta+|Du|^2)^{\frac{\gamma-2}{2}}\rho\,dxdt-\widetilde{C}_{H,4}\int_0^\tau\int_{\R^n}t^2\rho\,dxdt.
\]
On the other side, by the Young inequality we conclude
\[
\iint_{Q_\tau}2 t u_{ee}\rho\,dxdt\leq C_{H,4}\iint_{Q_\tau}t^2|D^2u|^2(\delta+|Du|^2)^{\frac{\gamma-2}{2}}\rho\,dxdt+\frac{1}{C_{H,4}}\iint_{Q_\tau}\rho(\delta+|Du|^2)^{-\frac{\gamma-2}{2}}\,dxdt.
\]
This implies
\[
\tau^2 u_{ee}\leq \frac{1}{C_{H,4}}\|(\delta+|Du|^2)^{\frac{2-\gamma}{2}}\|_\infty\tau+\widetilde C_{H,4}\frac{\tau^{3}}{3},
\]
which implies the assertion in the subquadratic case by letting $\delta \downarrow 0$. 
\end{proof}
Some remarks on the optimality of the constants are in order.
\begin{rem}
We observe that the modulus of semiconcavity $\frac{C}{t}$ in Theorem \ref{semic>0} cannot be in general improved. We show this for the model case of uniformly convex $H$, i.e. $\gamma=2$ in \eqref{H4}. Indeed, let $u$ be a solution to the problem 
\begin{equation*}
\begin{cases}
\partial_t u + H(Du) = 0
\\
u(x, 0)= |x|
\end{cases}
\end{equation*}
with $H$ satisfying $D^2_{pp}H(p)\xi \cdot \xi \geq \theta |\xi|^2 $ for some $\theta > 0$ (so $C_{H,4}=\theta$ and $\widetilde{C}_{H,4}=0$). Then, we have that $L$,   the Fenchel conjugate of $H$, is semiconcave with modulus $\frac{1}{\theta}$, and, from the Hopf-Lax formula, we have
\begin{multline*}
u(x+z, t) + u(x-z,t) - 2u(x, t) \leq t\left(L\left(\frac{x+z-y}{t}\right) + L\left(\frac{x-z-y}{t}\right) -2 L\left(\frac{x-y}{t}\right) \right)
\\
\leq\ \frac{t}{\theta} \left|\frac{z}{t} \right|^2 = \frac{1}{\theta t}|z|^2.
\end{multline*}
It is easy to see that the semiconcavity estimate in Theorem \ref{semic>0} now reads
\[
u_{ee}\leq \frac{1}{\theta t}.
\]
\end{rem}

\begin{rem}
In the subquadratic case, when $\widetilde{C}_{H,4}=0$, we recover by a different method estimate 3.2 in Proposition 3.2 of \cite{BKL}.
\end{rem}

\begin{rem}\label{conlaws}
When $n=1$, equation \eqref{hjintro} reduces to 
\[
\partial_t u-\eps u_{xx}+H(u_x)=0.
\]
In the special case $H(u_x)=|u_x|^\gamma$, one has that $U=u_x$ solves the regularized conservation law
\[
\partial_t U-\eps U_{xx}+(|U|^\gamma)_x=0.
\]
The estimate in Theorem \ref{semic>0} leads to the Oleinik-type one-side Lipschitz estimate, cf. \cite{EVZ},
\[
U_x\leq \frac{C}{t}\|U\|_\infty^{2-\gamma}.
\]
In general dimension $n$ one can apply a similar duality argument as that in Theorem \ref{semic>0} to the multidimensional conservation law $\partial_t u+\mathrm{div}(F(u))=\eps\Delta u$ with flux $F:\R\to\R^n$,  written in nondivergence form
\[
\partial_t u+F'_i(u)u_{x_i}=\eps\Delta u,
\]
obtaining an estimate as that in \cite{Hoff}. Some other related results by duality for multidimensional scalar conservation laws can be found in Section 7 of \cite{Evansadjoint}.
\end{rem}

\subsection{Second-order $L^p$ one-side bounds}\label{sec;semiLp}
In this section, we extend the $L^\infty$ bounds on second-order derivatives of the two previous subsections to $L^p$ bounds on the same quantities.
\begin{thm}\label{aho}
Assume that $H\in W^{1,\infty}_{\mathrm{loc}}(\R^n)$ satisfies \eqref{H1}-\eqref{H4}
and let $u_0:\R^n\to\R$ be $L^p_x$-SSH and $f$ be $L^1_t(L^\infty_x)$-SSH. Then, any classical solution to \eqref{hjintro} satisfies the one-side bound
\begin{equation}\label{semp}
\|(\Delta u_\eps)^+(t)\|_{L^p_{\mathrm{loc}}(\R^n)}\leq C,\ t\in(0,T).
\end{equation}
\end{thm}
\begin{proof}
The proof is the same as that in Theorem \ref{semiconc1}, the only difference being that we consider the solution of the adjoint problem
\[
\begin{cases}
-\partial_t \rho-\eps\Delta \rho-\mathrm{div}(D_pH(x,Du)\rho)=0&\text{ in }\R^n\times(0,\tau),\\
\rho(x,\tau)=\rho_\tau(x)&\text{ in }\R^n.
\end{cases}
\]
with $\rho_\tau\in C_c^\infty(\R^n)$, $\rho_\tau\in L^1(\R^n)\cap L^{p'}(\R^n)$, $\rho_\tau\geq0$ and $\|\rho_\tau\|_{L^1(\R^n)\cap L^{p'}(\R^n)}=1$, $p'>1$. This implies that
\[
\int_{\R^n}\rho_\tau(x)\,dx\leq 1.
\] 
Moreover, using the same localization argument of Lemma \ref{well}, we have
\begin{equation}\label{conserv}
\int_{\R^n}\rho(x,t)\,dxdt\leq 1.
\end{equation}
 Hence the proof continues along the same lines of that in Theorem \ref{semiconc1}. Note that by duality one gets a global estimate of $[(u_\eps)_{ee}]^+$ in the space $L^\infty(\R^n)+L^p(\R^n)$, which is embedded into $L^p_{\mathrm{loc}}(\R^n)$.
\end{proof}

\section{Quantitative properties of Hamilton-Jacobi equations}

\subsection{A survey on the rate of convergence for the vanishing viscosity approximation of Hamilton-Jacobi equations}
It is well-known that the viscosity solution to the first order Hamilton-Jacobi equation \eqref{hjintroinv} can be  obtained as the limit as $\eps\to0$ of the solutions to \eqref{hjintro}, see e.g. Chapter VI in \cite{BCD}. This limiting procedure is indeed fundamental to select   a solution of the first-order problem, as in general uniqueness for a.e. solutions  is not always expected, cf. Example in Section 3.3.3 of \cite{EvansBook}. The uniform convergence of $u_\eps$ to $u$ has been proved in Theorem 3.1 in \cite{CEL} and Theorem VI.3.1-3.2 in \cite{BCD}. We note that the global convergence requires extra regularity hypotheses on the solution, as discussed in Chapter VI of \cite{BCD}. Moreover, one can prove more refined quantitative properties such as the rate of convergence (with respect to $\eps$) of the vanishing viscosity process. The first results in this direction appeared in \cite{K65} for stationary problems with $H$ convex, and later refined in \cite{CL84}  and \cite{Souganidis} using doubling variables methods: for $W^{1,\infty}_{\mathrm{loc}}(\R^n)$ Hamiltonians, not necessarily convex, and an initial datum $u_0\in W^{1,\infty}(\R^n)$, it is proved the following rate for Lipschitz viscosity solutions
\[
\sup_{\R^n\times[0,T]}|u_\eps-u|\leq c\sqrt{\eps},
\] 
where $c$ depends on $H$, $u_0$ and $T$. Another proof of such a rate uses smoothing arguments through sup-inf convolutions (which are semiconvex-semiconcave), cf. p.76 of \cite{Calder}. The same rate for Lipschitz solutions has been proved via the adjoint method in \cite{Evansadjoint}, see also \cite{Tran2011}, \cite{GPV} and Theorem \ref{ev1}. We further emphasize that the $\mathcal{O}(\sqrt{\epsilon})$ rate is in general optimal, as the following example with $H=0$ (i.e. when there is no control) adapted from \cite{PerthameSanders} shows:
\begin{ex}\label{exrate1}
The function
\[
u_\eps(x)=\sqrt{\eps}\frac{ \cosh\left(\frac{x-1/2}{\sqrt{\eps}}\right)}{\sinh\left(\frac{1}{2\sqrt{\eps}}\right)}
\]
solves
\[
-\eps u''_\eps(x)+u_\eps(x)=0 \quad\text{in } (-1,1)
\]
with boundary conditions $u(0)=u(1)=\cosh(1/\sqrt{\eps})/\sinh(1/\sqrt{\eps})$ and
\[
|u_\eps-u|\leq C\sqrt{\eps}
\]
where $u\equiv0$ is the solution of the  problem with $\eps=0$ and null boundary datum. We should emphasize here that the boundary condition depends on $\eps$, and a more general example can be found in the recent paper \cite{Tran+altri}.
\end{ex}
The rate of convergence is sensitive to the regularity assumptions on the solution and on the Hamiltonian. Indeed, if one assumes that $u\in C^{0,\alpha}(\R^n)$, $\alpha\in(0,1]$, still with locally Lipschitz Hamiltonians, the rate becomes
\[
\sup_{\R^n\times[0,T]}|u_\eps-u|\leq c\eps^{\frac{\alpha}{2}},
\]
see \cite[Theorem 3.2]{BCD}. The rate would become slower if the Hamiltonian is  less regular, say only H\"older continuous  with exponent $\beta\in(0,1)$, and it would depend also on this new parameter. This latter point can be seen by a direct inspection of the proof in \cite[Theorem 3.2]{BCD}. 
Nonetheless, sometimes better rates are expected under additional assumptions, as in the next

\begin{ex}\label{ex2rate} Following \cite{Calder}, one can observe that the solution to
\[
-\eps u''_\eps(x)+|u'_\eps(x)|=1\text{ for }x\in(-1,1)\subset\R
\]
satisfying $u_\eps(-1)=u_\eps(1)=0$ is
\[
u_\eps(x)=1-|x|-\eps(e^{-\frac{|x|}{\eps}}-e^{-\frac1\eps})
\]
and $|u-u_\eps|\leq C\eps$, where $u(x)=1-|x|$ is the  viscosity solution to same problem with $\eps=0$. The difference with Example \ref{exrate1} is related to the additional properties satisfied by the solution of this problem. Indeed, such solutions are semi-superharmonic with a constant independent of $\eps$ or, better, semiconcave.
\end{ex}
In general, knowing  that $\Delta u_\eps\leq C$ (note that this condition is much weaker than the semiconcavity condition) independently of $\eps>0$, one can show  the one-side rate
\[
u_\eps-u\leq c\eps.
\]
The previous bound holds  for nonconvex, locally Lipschitz Hamiltonians, but it is conditional to the unilateral bound on $u_\eps$ which  usually requires convexity type assumptions. This improved rate has been first proved in Section 11 of \cite{L82book} using probabilistic methods under the assumption that $H$ is convex and the initial datum is SSH (i.e. $\Delta u_0\leq c_0$), and in \cite{BCD} using techniques from viscosity solution  theory.\\
 A related two-side $\mathcal{O}(\eps)$ rate has been proved by S.N. Kruzhkov in Lemma 2 of \cite{K65} for semiconcave solutions in $L^1$ and $L^\infty$, and by C.-T. Lin and E. Tadmor in $L^1$-norms under the assumption that $u$ is semiconcave stable (i.e. $L^1_tL^\infty_x$ semiconcave in our notation), $H$ is uniformly convex and in the case of periodic boundary conditions. Both these works exploit   duality arguments. Nonetheless, the bound from below on $u_\eps-u$ in sup-norm remains an open problem.\\
  Moreover, P.-L. Lions proved the convergence of $u_\eps$ to $u$ in $L^p$ for any $p$ in Chapter 6 of \cite{L82book}, so it is natural to determine the rate of convergence (possibly both for the positive and negative part of $u_\eps-u$) in Lebesgue norms. 
 We also mention that H. V. Tran \cite[Theorem 1.43]{TranBook} proved a   rate of order $\mathcal{O}(\eps)$ when the Hamiltonian is uniformly convex.\par

 Recent works have considered the problem of establishing the rate of the vanishing viscosity process in the context of Mean Field Games. The work \cite{TangZhang} proved the rate of convergence of the vanishing viscosity process both in the case of local and nonlocal coupling among the equations using duality methods, while \cite{DDJ} shows that the convergence problem in mean field control can be reduced to a problem of vanishing viscosity for finite dimensional Hamilton-Jacobi equations, as studied in the present paper.\par
 
 In this section we provide a unifying method for proving rates of convergence in any $L^p$ norm $1\leq p\leq\infty$ using duality methods and properties of transport equations, extending all the previous results under weaker assumptions on $H$. We further mention that the approach is flexible enough to cover various boundary conditions (periodic, Cauchy-Dirichlet, Neumann, whole space,...) as well as stationary problems. We will also give precise results on the size of the estimates taking care of the constants in the bounds.

  \subsection{Rate of convergence: the non-compact case}\label{sec;rate1}
  
\subsubsection{$L^\infty$ rate of convergence}
We consider, for simplicity, the viscous Cauchy problem
\begin{equation}\label{hjsemi}
\begin{cases}
\partial_t u_\eps-\eps\Delta u_\eps+H(Du_\eps)=f(x,t)&\text{ in }Q\\
u(x,0)=u_0(x)&\text{ in }\R^n.
\end{cases}
\end{equation}
and the first-order equation
\begin{equation}\label{first}
	\begin{cases}
		\partial_t u+H(Du)=f(x,t)&\text{ in }Q\\
		u(x,0)=u_0(x)&\text{ in }\R^n.
	\end{cases}
\end{equation}
From now on, we will mainly consider $H\in W^{1,\infty}_{\mathrm{loc}}(\R^n)$ and exploit the Lipschitz estimate
\[
\|Du\|_{L^\infty(Q)}\leq \|Du_0\|_{L^\infty(\R^n)}+T\|Df\|_{L^\infty(Q)}\quad \text{for $\eps\ge0$}
\]
to make the gradient of $u$ globally bounded and ensure the conservation of mass for Fokker-Planck equations. Lipschitz estimates are in general known under rather general conditions of coercivity, an example being \eqref{H1}, those in Corollary 4.1 p.100 in \cite{L82book}, Section 8 in \cite{Barles} or Chapter 1 in \cite{TranBook}.\\
We start proving a two-side rate of convergence for the vanishing viscosity of Lipschitz solutions. The result is already known from \cite{Evansadjoint}, we only slightly reword the proof and take care of the constants in the estimates to compare it with the corresponding results obtained in \cite{L82book}, see also \cite[p. 207]{FlemingGame} and Remark \ref{constants}.
\begin{thm}\label{ev1}
Let $H\in W^{1,\infty}_{\mathrm{loc}}(\R^n)$, $u_0\in W^{1,\infty}(\R^n)$ and $u_\eps$, $u_\eta$ be two solutions to \eqref{hjsemi} with $f=0$. Then
\[
\|u_\eps-u_\eta\|_{L^\infty(Q)}\leq \sqrt{2nT}(\sqrt{\eps}-\sqrt{\eta})\|Du_0\|_{L^\infty(\R^n)},\ \forall \eps\geq\eta\geq0.
\]
Moreover $u_\eps$ converges in $L^\infty(\R^n)$ to the viscosity solution $u\in W^{1,\infty}(Q)$ of the first-order Hamilton-Jacobi equation
\eqref{first} and we have the rate 
\[
\|u_\eps-u\|_{L^\infty(Q)}\leq \sqrt{2nT}\|Du_0\|_{L^\infty(\R^n)}\sqrt{\eps}.
\]
\end{thm}
\begin{proof}
We can assume that $H\in W^{1,\infty}(\R^n)$ by the global Lipschitz estimate in Lemma \ref{grad}. Indeed, one can consider (since $f=0$), setting $R=\|Du_0\|_\infty$, the truncated Hamiltonian
\[
\widetilde{H}(p)=H(p)\text{ if }|p|\leq R,\ \widetilde{H}(p)=H\left(\frac{R}{|p|}p\right)\text{ if }|p|\geq R,
\]
and argue with $\widetilde{H}\in W^{1,\infty}(\R^n)$ instead of $H$, which is only locally Lipschitz. 
We first estimate
\[
2\eps\iint_{Q}|D^2u_\eps|^2\rho\,dxdt\leq \|Du_0\|_{L^\infty(\R^n)}^2.
\]
We use the B\"ochner's identity to find for $g=|Du_\eps|^2$
\[
\partial_t g-\eps\Delta g+2\eps|D^2u_\eps|^2+D_pH(Du_\eps)\cdot Dg=0.
\]
We test the above equation with the adjoint variable $\rho$ solving 
\begin{equation}\label{adjrate1}
\begin{cases}
\partial_t \rho-\eps\Delta \rho-\mathrm{div}(b(x,t)\rho)=0&\text{ in }\R^n\times(0,\tau),\\
\rho(x,\tau)=\rho_\tau(x)&\text{ in }\R^n.
\end{cases}
\end{equation}
with $b(x,t)=D_pH(Du_\eps)$, $\rho_\tau\in C_c^\infty(\R^n)$, $\rho_\tau\geq0$, $\rho_\tau\in L^1(\R^n)$, $\|\rho_\tau\|_{L^1}=1$, to find
\[
2\eps\iint_{Q}|D^2u_\eps|^2\rho\,dxdt\leq -\int_{\R^n}g(\tau)\rho(\tau)\,dx+\int_{\R^n}g(0)\rho(0)\,dx\leq  \|Du_0\|_{L^\infty(\R^n)}^2.
\]
We now consider the equation satisfied by $z=\frac{\partial u_\eps}{\partial \eps}$ (or alternatively arguing with the finite difference $z_{\eta}=\frac{u_{\eps+\eta}-u_\eps}{\eta}$ and then sending $\eta\to0$), that is,
\[
\partial_t z-\eps\Delta z+D_pH(Du_\eps)\cdot Dz=\Delta u_\eps.
\]
By duality, using that $\rho\geq0$ and $z(0)=0$, along with the conservation of mass $\int_{\R^n}\rho\,dx=1$ (note that the drift is now globally bounded) and the Cauchy-Schwarz inequality, we obtain
\begin{align*}
\int_{\R^n}z(\tau)\rho_\tau(x)\,dx&= \iint_{Q}\Delta u_\eps\rho\,dxdt\leq \sqrt{n}\iint_{Q}|D^2u_\eps|\rho\,dxdt\\
&\leq \sqrt{n}\left(\iint_{Q}|D^2u_\eps|^2\rho\,dxdt\right)^\frac12\left(\iint_Q \rho\,dxdt\right)^\frac12\\&\leq \sqrt{\frac{nT\|Du_0\|_{L^\infty(\R^n)}^2}{2\eps}}\leq \sqrt{\frac{nT}{2\eps}}\|Du_0\|_{L^\infty(\R^n)}.
\end{align*}
The estimate on the negative part is similar. This gives
\[
|z(\tau)|\leq \sqrt{\frac{nT}{2\eps}}\|Du_0\|_{L^\infty(\R^n)}
\]
and hence for $\eps_2>\eps_1>0$
\[
\|(u_{\eps_2}-u_{\eps_1})(\tau)\|_{L^\infty(\R^n)}\leq \sqrt{2nT}\|Du_0\|_{L^\infty(\R^n)}(\sqrt{\eps_2}-\sqrt{\eps_1}).
\]
\end{proof}

\begin{rem}\label{constants}
Theorem \ref{ev1} provides an explicit dependence of the constant in the vanishing viscosity process. In contrast, the earlier results from \cite{CL84} provided an implicit and less precise constant independent of the dimension $ n$, but dependent on the Lipschitz constants of the data $ u_0$, $ H$, as well as $ T$. In this case, we require $H$ to be locally Lipschitz, and the size of the estimate depends on the dimension $ n$, without explicit dependence on the Lipschitz constant of $ H$. A similar estimate for Lipschitz solutions was stated, without proof, in Proposition 11.2 of \cite{L82book}.
\end{rem}

\begin{rem}
The rate of convergence in Theorem \ref{ev1} can be proved also for equations with a globally Lipschitz right-hand side $f$, but the estimates will depend also on $\|Df\|_\infty$. Indeed one would have
\begin{multline*}
\eps\iint_{Q}|D^2u_\eps|^2\rho\,dxdt\leq -\int_{\R^n}g(\tau)\rho(\tau)\,dx+\int_{\R^n}g(0)\rho(0)\,dx+\|Df\|_{L^\infty(Q)}\|Du_\eps\|_{L^\infty(Q)} T\\
\leq \frac{\|Du_0\|_{L^\infty(\R^n)}^2}{2}+\|Du_\eps\|_{L^\infty(Q)}^2+T\|Du_\eps\|_{L^\infty(Q)}\|Df\|_{L^\infty(Q)}.
\end{multline*}
One can also directly require $H\in W^{1,\infty}(\R^n)$ since the conservation of mass for the adjoint problem continues to hold. Under the assumptions of \cite{Tran2011}, $H$ can also be considered as a function of $x$ and $t$, and the proof above demonstrates that the same rate holds.
\end{rem}

We now turn to SSH solutions, and propose a new proof by a PDE method  of a result obtained by P.-L. Lions in \cite{L82book} through a related probabilistic argument. Differently from \cite{L82book}, we require on $u_\eps$ the weaker condition $L^1_t(L^\infty_x)$-SSH.

\begin{thm}\label{inftyrate1s}
Let $H\in W^{1,\infty}_{\mathrm{loc}}(\R^n)$, $u_0\in W^{1,\infty}(\R^n)$,  $u_\eps$ be a $L^1_t(L^\infty_x)$-SSH solution to \eqref{hjsemi} and $u_\eta$ another solution of \eqref{hjsemi} with viscosity $\eps$ replaced by $\eta$ and the same initial condition as $u_\eps$. Then
\[
\|(u_\eps-u_\eta)^+\|_{L^\infty(Q)}\leq \|(\Delta u_\eps)^+\|_{L^1(0,T;L^\infty(\R^n))}(\eps-\eta),\ \eps\geq\eta\geq0.
\]
\end{thm}

\proof
We start again with the difference $w=u_\eps-u_\eta$ satisfying
\[
\partial_t w-\eta\Delta w+H(Du_\eps)-H(Du_\eta)=(\eps-\eta)\Delta u_\eps.
\]
As above, we have
\[
\partial_t w-\eta\Delta w+\left(\int_0^1 D_pH(sDu_\eps+(1-s)Du_\eta)\,ds\right)\cdot Dw= (\eps-\eta)\Delta u_\eps.
\]
Using the solution $\rho\geq0$ to \eqref{adjrate1} and arguing by duality we obtain
\[
\int_{\R^n}w(\tau)\rho_\tau(x)\,dx\leq\int_{\R^n}w(0)\rho(0)+(\eps-\eta)\int_0^\tau\|(\Delta u_\eps)^+\|_{L^\infty(\R^n)}\int_{\R^n}\rho\,dxdt.
\]
Since $\int_{\R^n}\rho\,dx=1$ due to the standing assumptions on $H$ (and using the same truncation argument of Theorem \ref{ev1}) and the fact that $w(0)=0$, we get
\begin{equation*}
\|(u_\eps-u_\eta)^+\|_{L^\infty(\R^n)}\leq \|(\Delta u_\eps )^+\|_{L^1(0,T;L^\infty(\R^n))}(\eps-\eta)\eqno\square
\end{equation*}

\begin{rem}
Theorem \ref{inftyrate1s} can be extended to more general Hamiltonians depending also on $(x,t)$ as soon as the regularity of $D_pH$ ensures the well-posedness and the conservation of mass for transport equations with degenerate diffusion, see e.g. \cite{LBL}. Under such assumptions, we would conclude the same rate of convergence.
\end{rem}
Combining the previous result with the second-order bounds of Theorem \ref{semiconc1} we get the following one-side $\mathcal{O}(\eps)$ rate of convergence.
\begin{cor}\label{rateSSH}
Assume that  $H\in W^{1,\infty}_{\mathrm{loc}}(\R^n)$ satisfies \eqref{H1} and \eqref{H4}, $f$ is $L^1_t(L^\infty_x)$-SSH and $u_0$ is $L^\infty_x$-SSH and Lipschitz continuous. Then, the unique solution $u_\eps$ of \eqref{hjsemi} converges to the unique bounded viscosity solution $u\in W^{1,\infty}(Q)$ of \eqref{first}.
In addition, we have the following bound for all $t\in[0,T]$
\[
\|(u_\eps-u_\eta)^+(t)\|_{L^\infty(\R^n)}\leq (\|(\Delta f)^+\|_{L^1(0,T;L^\infty(\R^n))}+\|(\Delta u_0)^+\|_{L^\infty(\R^n)})T(\eps-\eta),\ \eps\geq\eta\geq0.
\]
Moreover, for $f=0$, we have  the two-side rate 
\[-\sqrt{2nT}\|Du_0\|_{L^\infty(\R^n)}\sqrt{\eps}\leq u_\eps-u\leq \|(\Delta u_0)^+\|_{L^\infty(\R^n)}T\eps.  \]
\end{cor}

\begin{proof}
The convergence of $u_\eps$ towards $u$ has been already discussed in the introduction. We prove the bound on $u_\eps-u_\eta$. Since $\Delta f\leq c(t)$ with $c\in L^1(0,T)$, this implies by Theorem \ref{semiconc1} and Remark \ref{Hconvex} the following bound
\[
\Delta u_\eps\leq \|(\Delta f)^+\|_{L^1(0,T;L^\infty(\R^n))}+\|(\Delta u_0)^+\|_{L^\infty(\R^n)}.
\]
Then, the result follows immediately by Theorem \ref{inftyrate1s} using that
\[
\|(\Delta u_\eps)^+\|_{L^1(0,T;L^\infty(\R^n))}\leq \|(\Delta u_\eps)^+\|_{L^\infty(Q)}T\leq (\|(\Delta f)^+\|_{L^1(0,T;L^\infty(\R^n))}+\|(\Delta u_0)^+\|_{L^\infty(\R^n)})T.
\]
When $f=0$ the second statement follows from the above estimate combined with Theorem \ref{ev1}.
\end{proof}
\begin{rem}\label{weaken}
If $f=0$ and $H=H(p)\in W^{1,\infty}_{\mathrm{loc}}(\R^n)$ is convex, one has by Corollary \ref{rateSSH}
\[
\|(\Delta u)^+(t)\|_{L^\infty(\R^n)}\leq \|(\Delta u_0)^+\|_{L^\infty(\R^n)}.
\]
Consequently, the estimate for the one-side rate of convergence becomes
\[
u_\eps-u_\eta\leq \|(\Delta u_0)^+\|_{L^\infty(\R^n)}T(\eps-\eta),\ \eps\geq\eta\geq0.
\]
This is the same estimate stated by P.-L. Lions in Proposition 11.2 of \cite{L82book}. The corresponding estimate for the stationary problem has been proved in Section 6.2 of \cite{L82book} via probabilistic methods. Our proofs and the results in the previous Theorems \ref{inftyrate1s} and Corollary \ref{rateSSH} are new and valid for possibly nonconvex Hamiltonians.
\end{rem}

\subsubsection{$L^p$ rate of convergence}
In the next proposition, we extend the $L^\infty$ estimate in Theorem \ref{ev1} to $L^p_{\mathrm{loc}}$ norms for Lipschitz continuous solutions.
\begin{thm}\label{prate1}
Let $H\in W^{1,\infty}_{\mathrm{loc}}(\R^n)$, $u_0\in W^{1,\infty}(\R^n)$ and $u_\eps$ be a solution to \eqref{hjsemi}. Then, $u_\eps$ converges in $L^\infty(Q)$ to the viscosity solution $u\in W^{1,\infty}(Q)$ of the first-order Hamilton-Jacobi equation \eqref{first}. We have the rate
\[
\|u_\eps-u\|_{L^p_{\mathrm{loc}}(Q)}\leq C\sqrt{\eps}, \quad 1 \leq p < \infty,
\]
where $C$ depends on $n,T,\|Du_0\|_{L^\infty(\R^n)},p$.
\end{thm}
\begin{proof}
The proof is the same as that in Theorem \ref{ev1}, but we have to introduce the adjoint problem \eqref{adjrate1} with terminal data $\rho_\tau\in L^1\cap L^{p'}$, as in Theorem \ref{pSSHrate}.
\end{proof}

We next prove a one-side estimate by duality for SSH solutions on the whole space.
\begin{thm}\label{pSSHrate}
Let $H\in W^{1,\infty}_{\mathrm{loc}}(\R^n)$, $u_\eps$ be a $L^1_t(L^\infty_x)$-SSH solution to \eqref{hjsemi}, and $u_\eta$ be another solution (with viscosity $\eps$ replaced by $\eta$) with the same initial condition. Then, there exists a constant $C>0$ such that for all $t\in(0,T)$
\[
\|(u_\eps-u_\eta)^+(t)\|_{L^p_{\mathrm{loc}}(\R^n)}\leq C\|(\Delta u_\eps)^+\|_{L^1(0,T;L^\infty(\R^n))}(\eps-\eta),\ \eps\geq\eta\geq0,\ p\geq1.
\]
\end{thm}
\begin{proof}
The difference $w=u_\eps-u_\eta$ satisfies
\[
\partial_t w-\eta\Delta w+H(Du_\eps)-H(Du_\eta)=(\eps-\eta)\Delta u_\eps, w(0)=0.
\]
As in Theorem \ref{inftyrate1s} we linearize the equation and then introduce the adjoint problem \eqref{adjrate1} with $b(x,t)=-\int_0^1 D_pH(sDu_\eps+(1-s)Du_\eta)\,ds$ with terminal datum $\rho_\tau\in C_0^\infty(\R^n)$, $\rho_\tau\geq0$, $\rho_\tau\in L^1(\R^n)\cap L^{p'}(\R^n)$, $\|\rho_\tau\|_{L^1\cap L^{p'}}=1$. Note that, arguing as in Theorem \ref{aho}, we have
\begin{equation}\label{L1p}
\int_{\R^n}\rho_\tau(x)\,dx\leq 1\quad \textrm{and}\quad  \int_{\R^n}\rho(x,t)\,dx\leq 1,\ \forall t\in[0,\tau).
\end{equation} 
Moreover, $\rho\geq0$ by the maximum principle. By duality, we obtain
\[
\int_{\R^n}w(\tau)\rho_\tau(x)\,dx\leq\int_{\R^n}w(0)\rho(0)+(\eps-\eta)\int_0^T\|(\Delta u_\eps)^+\|_{L^\infty(\R^n)}\int_{\R^n}\rho\,dxdt.
\]
Using that $w(0)=0$ and \eqref{L1p},  we have
\[
\|(u_\eps-u_\eta)^+(t)\|_{L^\infty(\R^n)+L^p(\R^n)}\leq C_1\|(\Delta u_\eps)^+\|_{L^1(0,T;L^\infty(\R^n))}(\eps-\eta)
\]
Moreover, appealing to the embedding $L^\infty(\R^n)+L^p(\R^n)\hookrightarrow L^p_{\mathrm{loc}}(\R^n)$, we have that, for all $K\subset\subset\R^n$, there exists a constant $C_2$ depending on $C_1,p,K$ such that
\[
\|(u_\eps-u_\eta)^+(t)\|_{L^p(K)}\leq C_2\|(\Delta u_\eps)^+\|_{L^1(0,T;L^\infty(\R^n))}(\eps-\eta).
\]
\end{proof}

\begin{rem}
One can remove the bound on $\|(\Delta u_\eps)^+\|_{L^1(0,T;L^\infty(\R^n))}$ in Theorem \ref{prate1}, as in Corollary \ref{rateSSH}, using the one-side bounds in Theorem   \ref{semiconc1} and obtain a more precise estimate.
\end{rem}

The next result shows instead that a one-side rate in $L^p$ for $p>1$ holds globally in $\R^n$ under the additional assumption that $u_\eta$ is semiconcave  and $H$ fulfills \eqref{H4}. These further assumptions  are fundamental to apply the stability estimates for transport equations in Theorem \ref{Krylov}. However, we weaken the requirement on $u_\eps$, which will now be assumed in $L^p$-SSH, $p>1$. Note that uniqueness and stability for $L^p$-SSH solutions require the restriction $p\geq n$, cf. Remark 3.6 of \cite{L82book}.
\begin{thm}
Let $H\in W^{1,\infty}_{\mathrm{loc}}(\R^n)$ be convex and satisfying \eqref{H4}, $u_\eps$ be a $L^1_t(L^p_x)$-SSH solution to \eqref{hjsemi}, and $u_\eta$ be another $L^1_t(L^\infty_x)$-semiconcave solution (with viscosity $\eps$ replaced by $\eta$) having the same initial condition. Then, there exists a constant $C>0$
\[
\|(u_\eps-u_\eta)^+(t)\|_{L^p(\R^n)}\leq C\|(\Delta u_\eps)^+\|_{L^1(0,T;L^p(\R^n))}(\eps-\eta),\ \eps\geq\eta\geq0,\ p>1.
\]
where $C$ depends on the data of the problem and also on $\|(D^2u_\eta)^+\|_{L^1_t(L^\infty_x)}$.
\end{thm}
\begin{proof}
The proof is similar to the previous result, since the difference $w^+=(u_\eps-u_\eta)^+$ satisfies the inequality
\[
\partial_t w^+-\eta\Delta w^++D_pH(Du_\eta)\cdot Dw\chi_{\{w>0\}}\leq (\eps-\eta)\Delta u_\eps \chi_{\{w>0\}}, w(0)=0.
\]
Now, we consider the adjoint problem 
\begin{equation*}
\begin{cases}
\partial_t \rho-\eta\Delta \rho-\mathrm{div}(D_pH(Du_\eta)\rho \chi_{\{w>0\}})=0&\text{ in }\R^n\times(0,\tau)\\
\rho(x,\tau)=\rho_\tau(x)&\text{ in }\R^n.
\end{cases}
\end{equation*}
with terminal datum $\rho_\tau:=(w^+(\tau))^{p-1}/\|w^+\|^{p-1}_{L^p}$ that belongs only to $L^{p'}$ having $\|\rho_\tau\|_{p'}\leq1$ (and not on the intersection $L^1\cap L^{p'}$). This yields a (global) bound on $(u_\eps-u_\eta)^+(t)\in L^p(\R^n)$ for all $t\in(0,T)$. However, one has to estimate the term involving $\Delta u_\eps$ on the right-hand side of the equation satisfied by $w$ by the H\"older's inequality as follows:
\[
(\eps-\eta)\int_0^T\|(\Delta u_\eps)^+\|_{L^p(\R^n)}\left(\int_{\R^n}|\rho|^{p'}\right)^{\frac{1}{p'}}\,dxdt.
\]
Then, one applies the $L^r$ stability estimates in Theorem \ref{Krylov} with $r=p'$ to bound $\|\rho(t)\|_{p'}$ in terms of $\|(D^2 u_\eta)^+\|_{L^1_t(L^\infty_x)}$.
\end{proof}
\subsection{Rate of convergence: an improved estimate for compact state spaces}\label{sec;rate2}
We now address $L^p$ rates of convergence using duality methods, as initiated in \cite{LinTadmor}. In this setting we need to strengthen the requirement both on $u_\eps$ and $u_\eta$, and work in a compact state space, but we are able to get a two side control on the difference $u_\eps-u_\eta$ in any $L^p$ space and for semiconcave or semi-superharmonic solutions. 

The next is the main result of the section. It contains an estimate on the rate of convergence of the solution of the viscous equation $u_\eps$ towards the inviscid solution $u$ in the space $L^\infty_t(L^1_x)$ under the assumption that both of them are $L^1_t(L^\infty_x)$-semiconcave. It also provides a second rate of convergence in the stronger norm $L^\infty_t(L^p_x)$ at the expenses of assuming a two-side a priori bound on the solution.
\begin{thm}\label{rateL1}
Let $u_\eps, u_\eta$ be $L^1_t(L^\infty_x)$-semiconcave solutions of \eqref{hjsemi} with viscosity $\eps$ and, respectively, $\eta$. Let also $H\in W^{1,\infty}_{\mathrm{loc}}(\R^n)$ be such that \eqref{H4} holds. Then, there exists a constant $C>0$ such that
\[
\|u_\eps-u_\eta\|_{L^\infty(0,T;L^1(\T^n))}\leq C(\eps-\eta),\ \eps\geq\eta\geq0.
\]
where $C$ depends on the semiconcavity constant of $u_\eps$ and $u_\eta$. If, in addition, $-u_\eta$ is $L^1_t(L^\infty_x)$-SSH we conclude
\[
\|u_\eps-u_\eta\|_{L^\infty(0,T;L^p(\T^n))}\leq C(\eps-\eta),\ \eps\geq\eta\geq0,\;\; 1 \leq p \leq \infty,
\]
\end{thm}
\begin{proof}
Consider $w=u_\eps-u_\eta$ satisfying the equation
\[
\partial_t w-\eta\Delta w+H(Du_\eps)-H(Du_\eta)= (\eps-\eta)\Delta u_\eps, w(0)=0.
\]
We thus have
\[
\partial_t w-\eta\Delta w-b(x,t)\cdot Dw= (\eps-\eta)\Delta u_{\eps}, w(0)=0,
\]
where $b$ is the average velocity
\[
b(x,t)=-\int_0^1 D_pH(sDu_\eps+(1-s)Du_\eta)\,ds.
\]
We now recall that since $u_\eps,u_\eta$ are semiconcave stable we have, using \eqref{H4}, the following estimate for the average velocity
\[
\mathrm{div}(b)=-\int_0^1 s\sum_{i,j}D^2_{p_ip_j}H\partial_{x_ix_j}u_\eps+(1-s)\sum_{i,j}D^2_{p_ip_j}H\partial_{x_ix_j}u_\eta\,ds\geq -c(t)\in L^1.
\]

We start with the case $p=1$ as in Theorem \ref{pSSHrate}, by using the adjoint method. Here, we consider the adjoint problem
\begin{equation}\label{adjrate2}
\begin{cases}
\partial_t \rho-\eta\Delta \rho+\mathrm{div}(b(x,t)\rho)=0&\text{ in }\T^n\times(0,\tau),\\
\rho(x,\tau)=\rho_\tau(x)&\text{ in }\T^n,
\end{cases}
\end{equation}
but with terminal datum $\rho(\tau)=\mathrm{sgn}(w(\tau))$ on $\T^n$. Note that
\[
-1\leq\rho(\tau)\leq 1\implies \|\rho(\tau)\|_{L^\infty(\T^n)}\leq 1.
\]
Arguing by duality we get
\begin{multline*}
\int_{\T^n}w(\tau)\rho_\tau(x)\,dx=\int_{\R^n}w(0)\rho(0)\,dx+(\eps-\eta)\int_0^T\int_{\T^n}\Delta u_\eps\rho\,dxdt\\
\leq\|w(0)\|_{L^1(\T^n)}\|\rho(0)\|_{L^\infty(\T^n)}+ (\eps-\eta)\|\Delta u_\eps\|_{L^1(Q_\tau))}\|\rho\|_{L^\infty(Q_\tau)}.
\end{multline*}
We note that, since $u_\eps$ is smooth and $\|(\Delta u_\eps(t))^+\|_{L^\infty(\T^n)}\leq c(t)\in L^1$, we can get the estimate
\begin{multline*}
\int_{\T^n}|\Delta u_\eps(t)|\,dx\leq \int_{\T^n}|c(t)-\Delta u_\eps(t)|\,dx+c(t)\\
=\int_{\T^n}(c(t)-\Delta u_\eps(t))\,dx+c(t)=2c(t)+\int_{\T^n}(-\Delta u_\eps)\,dx=2c(t).
\end{multline*}
It implies
\begin{equation}\label{laplL1}
\iint_{Q_\tau}|\Delta u_\eps|\,dxdt\leq 2\int_0^\tau c(t)\,dt<\infty.
\end{equation}
It remains to estimate $\|\rho\|_{L^\infty(Q_\tau)}$ in terms of the final datum of the adjoint problem, since $w(0)=0$. A classical stability estimate for continuity equations that follows from the Gronwall's inequality or the Feynman-Kac formula \cite[Theorem 4.12]{Figalli} (see also p.710 in \cite{LinTadmor})  leads to
\[
\|\rho(t)\|_{L^\infty(\T^n)}\leq \|\rho(\tau)\|_{L^\infty(\T^n)}e^{\int_0^T\|(\mathrm{div}(b))^-\|_{L^\infty(\T^n)}},\ t\in[0,\tau).
\]
To prove the general case for $p=\infty$, we argue by duality as in Theorem \ref{inftyrate1s} to find the bound on $u_\eps-u_\eta$ from above. The bound from below follows noting that $z=u_\eta-u_\eps$ solves
\[
\partial_t z-\eps\Delta z+H(Du_\eta)-H(Du_\eps)=-(\eps-\eta)\Delta u_\eta,
\]
and hence
\[
\partial_t z-\eps\Delta z-b(x,t)\cdot Dz=-(\eps-\eta)\Delta u_\eta,
\]
where
\[
b(x,t)=-\int_0^1 D_pH(sDu_\eta+(1-s)Du_\eps)\,ds.
\]
We can now proceed again by duality using that $\Delta u_\eta\geq -c_\eta(t)$, and hence testing equation solved by $z$ against the solution of \eqref{adjrate2} having $\rho_\tau\in L^1$. This implies, since $\int_{\T^n}\rho(t)\,dx=1$ and $\rho\geq0$, the inequality
\[
\int_{\T^n}z(\tau)\rho_\tau(x)\,dx=(\eps-\eta)\iint_{Q}(-\Delta u_\eta)\rho\,dxdt\leq (\eps-\eta)\iint_{Q}c_\eta(t)\rho\,dxdt=(\eps-\eta)\int_0^\tau c_\eta(t)\,dt.
\]
The case $p\in(1,\infty)$ follows by the compactness of the flat torus or by interpolation.
\end{proof}

\begin{cor}
Under the assumptions of Theorem \ref{semiconc1}, let $u_\eps$ and $u_\eta$ be $L^1_t(L^\infty_x)$-semiconcave with viscosity $\eps$ and, respectively, $\eta$. Then
\[
\|u_\eps-u_\eta\|_{L^\infty(0,T;L^1(\T^n))}\leq C(\eps-\eta).
\]
where $C$ depends on $c_f(t)\in L^1(0,T)$, $c_0$ and the constants $C_{H,4}$ appearing in \eqref{H4}.
\end{cor}
\begin{rem}
Quite sharp assumptions ensuring that $u_\eta$ is in $W^{2,\infty}$ can be found in Proposition 7.1 of \cite{L82book}. In the second part of Theorem \ref{rateL1} for the $L^p$ case we are requiring only a control on the trace of the Hessian of $u_\eta$ from below, in addition to the semiconcavity of $u_\eps,u_\eta$.
\end{rem}
Under the assumptions ensuring the rate $\mathcal{O}(\sqrt{\eps})$ in $L^\infty$ norm in Theorem \ref{ev1} and those guaranteeing the $L^1$ rate of order $\mathcal{O}(\eps)$ in Theorem \ref{rateL1}, we can remove the two-side assumption on $u_\eta$ and obtain a two-side rate of convergence in $L^p$ norms, $p>1$.
\begin{cor}\label{prate}
Assume that $u_\eps,u$ are $L^1_t(L^\infty_x)$-semiconcave and Lipschitz continuous solutions of \eqref{hjsemi} and \eqref{first}. Then we have the $L^p$ rate of convergence
\[
\|u_\eps-u\|_{L^\infty(0,\tau;L^p(\T^n))}\leq C\eps^{\frac12+\frac{1}{2p}}.
\]
\end{cor}
\begin{proof}By interpolation, if $u_\eps,u$ are $L^1_t(L^\infty_x)$-semiconcave solutions with $u_0\in W^{1,\infty}$ and $D^2 u_0\leq C$, we have
\[
\|u_\eps-u\|_{L^\infty_t(L^p_x)}\leq \|u_\eps-u\|_{L^\infty_t(L^1_x)}^{\frac1p}\|u_\eps-u\|_{L^\infty_t(L^\infty_x)}^{1-\frac1p}\leq C\eps^{\frac12+\frac{1}{2p}}
\]
for any finite $p>1$.
\end{proof}
\begin{rem}
The estimate in Theorem \ref{rateL1} was stated for semiconvex solutions of some stationary equations in Theorem 4 of \cite{K75}, eq. (16), in $L^1$ norms with the slower order $\mathcal{O}(\eps^\nu)$, $\nu\in(0,1)$, and in eq. (17) of  \cite{K75} in $L^p$ norms for $p$ sufficiently large. The $L^1$ rate for semiconcave solutions is also the subject of Lemma 2 in \cite{K66d} for uniformly convex $H$.
\end{rem}

\begin{rem}\label{rateDu}
It is worth remarking that \eqref{laplL1} combined with the rates $\|u_\eps-u_\eta\|_{L^\infty}\leq C(\eps-\eta)$ or $\|u_\eps-u_\eta\|_{L^\infty}\leq C(\sqrt{\eps}-\sqrt{\eta})$ obtained in Theorems \ref{rateL1} and \ref{ev1} imply the  integral   estimates
\[
\esssup_{[0,T]}\int_{\T^n}|D(u_\eps-u_\eta)|^2\,dxdt\leq C(\eps-\eta)\text{ or }\esssup_{[0,T]}\int_{\T^n}|D(u_\eps-u_\eta)|^2\,dxdt\leq C(\sqrt{\eps}-\sqrt{\eta}).
\]
This slightly improves Remark 6.9 in \cite{L82book}, being valid for possibly nonconvex $H$ (i.e. under \eqref{H1}-\eqref{H4}). Indeed, an integration by parts and Lemma \ref{laplL1} give
\begin{align*}
\int_{\T^n}|D(u_\eps(t)-u_\eta(t))|^2\,dxdt&=-\int_{\T^n}(u_\eps(t)-u_\eta(t))\Delta(u_\eps(t)-u_\eta(t))\,dxdt\\
&\leq \|u_\eps(t)-u_\eta(t)\|_{L^\infty(\T^n)}\|\Delta(u_\eps(t)-u_\eta(t))\|_{L^1(\T^n)}\\
&\leq C\|u_\eps(t)-u_\eta(t)\|_{L^\infty(\T^n)}.
\end{align*}
This implies
\[
\|Du_\eps-Du\|_{L^\infty(0,T;L^2(\T^n))}\leq C\sqrt{\eps}.
\]
or
\[
\|Du_\eps-Du\|_{L^\infty(0,T;L^2(\T^n))}\leq C\eps^\frac14.
\]
S.N. Kruzhkov obtained in \cite{K66d} a $L^1$ rate of convergence of the gradient with order $\mathcal{O}(\eps^\frac12)$ using Gagliardo-Nirenberg inequalities and (ii) in Lemma \ref{laplL1}. The Gagliardo-Nirenberg inequality implies
\begin{align*}
\|Du_\eps(t)-Du_\eta(t)\|_{L^1(\T^n)}&\leq C_1(n)\|D^2u_\eps(t)-D^2u_\eta(t)\|_{L^1(\T^n)}^\frac12\|u_\eps(t)-u_\eta(t)\|_{L^1(\T^n)}^\frac12\\
&+C_2\|u_\eps(t)-u_\eta(t)\|_{L^1(\T^n)}.
\end{align*}
The estimate on $D^2u\in L^1(\T^n)$ shows that
\[
\|Du_\eps-Du\|_{L^\infty(0,T;L^1(\T^n))}\leq C\sqrt{\eps}.
\]
Also, we can use the $L^p$ rate of Corollary \ref{prate} to find by the Gagliardo-Nirenberg interpolation
\[
\|Du_\eps-Du\|_{L^\infty(0,T;L^\frac{2p}{p+1}(\T^n))}\leq C\eps^{\frac14+\frac{1}{4p}}.
\]
\end{rem}

\begin{rem}\label{system}
As noticed in Section 16.1 of \cite{L82book}, starting with a solution $u$ of
\[
\partial_t u+H(Du)=0\text{ in }Q,
\]
one obtains that $v=Du$ (with $v_i=u_{x_i}$) solves the hyperbolic quasilinear system
\[
\partial_t v_i+(H(v))_{x_i}=0\text{ in }Q.
\]
Existence, stability and further properties of solutions for such special systems can be thus obtained following the lines of Theorem 16.1 in \cite{L82book} or Theorem 8 of \cite{K67II}. In particular, S.N. Kruzhkov in Theorem 8 of \cite{K67II} obtained the convergence of the solution of the viscous quasilinear system 
\[
\partial_t v^\eps_i+(H(v^\eps))_{x_i}=\eps\Delta v^\eps_i\text{ in }Q
\]
towards the inviscid system solved by $v_i$, which arise from a Hamilton-Jacobi equation with uniformly convex $H$. In this case, as already observed in Remark \ref{conlaws}, the semi-superharmonic condition on $u$ becomes the classical Oleinik one-side Lipschitz condition on $v$ ensuring uniqueness of entropy solutions with convex fluxes, cf. \cite{EvansBook}. \\
By means of Remark \ref{rateDu} one obtains for instance a new $\mathcal{O}(\sqrt{\eps})$ rate in $L^1$ (or even $L^2$) for the convergence of $v^\eps$ to $v$. This kind of relation between Hamilton-Jacobi equations and hyperbolic systems has been used in \cite{K67II} and also recently in the context of Mean Field Games \cite{CecchinDelarue}, while $L^1$ rates for the vanishing viscosity approximation of hyperbolic systems were considered in \cite{Bressan} for $n=1$.
\end{rem}

\subsubsection{Extensions to stationary problems}
Throughout this section we briefly discuss how to extend the previous rate of convergence results for the stationary problem
\begin{equation}\label{hj_staz}
	\lambda u(x)+H(Du(x))=f(x)\text{ in }\R^n, \lambda>0.
\end{equation}
As before, we consider the regularized problem
\begin{equation}\label{HJstat}
-\eps\Delta u_\eps(x)+\lambda u_\eps(x)+H(Du_\eps(x))=f(x)\text{ in }\R^n.
\end{equation}
We state a model result that extends the rate of convergence for Lipschitz solutions obtained in Theorem 2.1 of \cite{Tran2011} to norms weaker than $L^\infty$. This holds for non-convex $H$.
\begin{thm}
Assume that $H,f\in W^{1,\infty}_{\mathrm{loc}}(\R^n)$. Let $u$, $u_\eps$ be a solution of \eqref{hj_staz} and, respectively, of \eqref{HJstat}. Then
\[
\|u_\eps-u\|_{L^\infty(\R^n)}\leq \frac{\sqrt{n}}{\lambda^2}\|Df\|_{L^\infty(\R^n)}\sqrt{\eps}.
\]
In addition, there exists a constant $C>0$ independent of $\eps$ such that
\[
\|u_\eps-u\|_{L^p_{\mathrm{loc}}(\R^n)}\leq C\sqrt{\eps}.
\]
\end{thm}
\begin{proof}
The proof can be done as in \cite{Tran2011}, using the transformation $w(x,t)=e^t u(x)$, which is a solution of the parabolic problem. One can then exploit the proof of the parabolic case and then go back to the elliptic case as in Theorem 2.7 of \cite{Tran2011}.
\end{proof}
One can formulate similar results with rate $\mathcal{O}(\varepsilon)$ up to $L^1$ for SSH solutions as it is done in the parabolic case. An example is the following
\begin{thm}
Assume that $H\in W^{1,\infty}_{\mathrm{loc}}(\R^n)$, $f$ being $L^\infty$-SSH. Assume that $u_\eps$ is $L^\infty$-SSH. Then
\[
u_\eps-u\leq \frac{1}{\lambda}\|(\Delta f)^+\|_{L^\infty(\R^n)}\eps
\]
In addition, for $p>1$ there exists a constant $C>0$ independent of $\eps$ such that
\[
\|(u_\eps-u)^+\|_{L^p_{\mathrm{loc}}(\R^n)}\leq C\|(\Delta f)^+\|_{L^\infty(\R^n)}\eps.
\]
\end{thm}

\subsection{Rate of convergence for numerical methods: the Godunov scheme}\label{sec;num}
In this part, we exploit the results of the previous sections in order to obtain 
a $L^1$ rate of convergence for   Godunov-type  approximation schemes for
Hamilton-Jacobi equations. This part improves the $L^1$ estimate obtained   in \cite[Theorem 2.3]{LinTadmor} where the same rate was obtained for uniformly convex Hamiltonians. \\
For simplicity, we consider  Hamilton-Jacobi equation of the type
\begin{equation}\label{num:HJ}
	\left\{	\begin{array}{ll}
		\partial_tu+H(Du)=0\qquad& \text{in}\,\mathbb{T}^n\times(0,\infty),\\
		u(x,0)=u_0(x)&\text{in}\,\mathbb{T}^n.
	\end{array}
	\right.	
\end{equation}
with $n=2$. 
We  fix a time grid $t^n=n\dt$, $n\in \mathbb{N}$,  and a  rectangular grid of cells of size $\D=\dx\times \dy$ which satisfies the non degeneracy condition
$0<c_0\le \dx/\dy\le C_0$ and  the CFL condition
$L_H \dt/\max\{\D x,\D y\}	 <1/4$ where $L_H$ is the Lipschitz constant of $H$.
A Godunov scheme reads a
\begin{equation}\label{num:Godunov}
	u^\D(\cdot,t)=
	\begin{cases}
		E(t-t^{n-1})u^\D(\cdot,t^{n-1})\quad& t\in(t^{n-1}, t^n)\\
		P^\D u^\D(\cdot, t^{n,-})&t=t^n
	\end{cases}
	\qquad n=1,2,\dots
\end{equation}
with $u^\D(\cdot, 0)= P^\D u_0(\cdot)$, 
where $E(\cdot)$ is the exact solution operator associated to the Hamilton-Jacobi equation \eqref{num:HJ}, $P^\D$ a projection operator on the grid and $u^\D(\cdot, t^{n,-})=E(t^n-t^{n-1})u^\D(\cdot,t^{n-1})$.\\
We need two preliminary results. The first one gives an estimate of the truncation error in terms of the $L^1$-norm of the error introduced by the projection operator $P^\D$
(see \cite[Lemma 2.1]{LinTadmor}).
\begin{lemma}
	Let $u^\D$ be a family of   functions given by the scheme \eqref{num:Godunov}. Then
	\begin{equation}\label{num:trunc_error}
		\|\partial_tu^\D +H(Du^\D)\|_{L^1_x}\le\frac{T}{\dt}
		\max_{0<t^n<T}\|(I-P^\D)u^\D(\cdot, t^{n,-})\|_{L^1_x}
	\end{equation}	
	($I$ denotes the identity operator).
\end{lemma}
The second lemma is a stability result that can be proved in the same way of the $L^1$ estimate in Theorem \ref{rateL1} via the adjoint method.
\begin{lemma}\label{stab}
Assume that $H$ satisfies \eqref{H1}-\eqref{H4}. For $i=1,2$, let $u_i$  be $L^1_t(L^\infty_x)$-semiconcave  solution  to
\begin{equation}\label{stability}
	\begin{cases}
		\partial_t u_i+H(Du_i)=f_i&\text{ in }\T^n\times (0,T)\\
		u_i(x,0)=u^i_0(x)&\text{ in }\T^n.
	\end{cases}
\end{equation}
Then 
\begin{equation}\label{God_est1}
\|(u_1-u_2)(t)\|_{L^1_x}\leq C (\|(u^1_0-u^2_0\|_{L^1_x}+\|f_1-f_2\|_{L^1_t(L^1_x)}),\quad t\in (0,T).
\end{equation}
\end{lemma}
Given $w:\mathbb{T}^2\times [0,T]\to\R$, $\xi\in\mathbb{R}^2$ with $|\xi|=1$ and $h>0$, we define second-order finite difference operator
\[
D^2_{h,\xi}w(x,t)=\frac{w(x+h\xi,t)+w(x-h\xi,t)-2w(x,t)}{h^2}
\]
and the norm 
\[
\normdiscrete{w(t)}:=\sup_{h>0,|\xi|=1}\|D^2_{h,\xi}w(x,t)\|_{L^1_x}.
\]
A family $\{\psi^\D\}$, $\D>0$, is said to be uniformly semiconcave in $L^1_t(L^\infty_x)$ if it  satisfies Definition \ref{semicmix} with the same function $k\in L^1_t(L^\infty_x)$ for any $\D$.
In the following result, we give an abstract $L^1$ estimate for the rate of convergence of Godunov schemes. 
\begin{thm}\label{num:theo_est}
Assume that $H$ satisfies \eqref{H1}-\eqref{H4}.	Let $u^\D$ be a family of  functions given by the scheme \eqref{num:Godunov} and assume that
	\begin{itemize}
		\item[(i)] For any $t>0$,
		\begin{equation}\label{num:consistent}
			\|(I-P^\D)u^\D(t)\|_{L^1_x}\le C\D^2 \normdiscrete{u^\D(t)},\quad t\in (0,T) .
		\end{equation}
		\item[(ii)] There exists a family of functions $\psi^\D$  such that
		\begin{align}
			&\text{$\psi^\D$ is  uniformly  semiconcave in $L^1_t(L^\infty_x)$ for $\D>0$} \label{num:nearby_1},\\
			&\|\partial_t(u^\D(t)-\psi^\D(t))\|_{L^1_x}+\|D(u^\D(t)-\psi^\D(t))\|_{L^1_x}\le C\D  \normdiscrete{\psi^\D(t)}.\label{num:nearby_2}
		\end{align}
		Then $u^\D$ converges to the viscosity solution $u$ of \eqref{num:HJ} and 
		\begin{equation}\label{num:est_Lp}
			\|u(t) -u^\D(t)\|_{L^1_{x}}\le C\D, \qquad t\in [0,T]
		\end{equation}
	\end{itemize}
\end{thm}
\begin{proof}
	By  \eqref{num:nearby_1}, it follows that	$\normdiscrete{\psi^\D(t)}\le C$, $t\in [0,T]$,  and by \eqref{num:nearby_2}	
	\begin{align*}
		&\|u( t)-u^\D(t)\|_{L^1_x}	\le \|u( t)-\psi^\D(t)\|_{L^1_x}+ \|\psi^\D(t)-u^\D( t)\|_{L^1_x}\\
		&\le \|u(  t)-\psi^\D( t)\|_{L^1_x}+  C\D  
	\end{align*}
		By Lemma \ref{stab} and \eqref{num:nearby_1}, we have for $t\in [0,T]$  
		\begin{equation}\label{num:proof1}
			\|u(  t)-\psi^\D( t)\|_{L^1_x}\le C\left(	\|u( 0)-\psi^\D(0)\|_{L^1_x}+\|\partial_t\psi^\D+H(D\psi^\D)\|_{L^1_x}\right).
	\end{equation}

	By \eqref{num:consistent} and \eqref{num:nearby_2}, we   estimate 
	\begin{equation}\label{num:proof2}
		\begin{split}
			\|u( 0)-\psi^\D( 0)\|_{L^1_x}&\le \|u( 0)-u^\D( 0)\|_{L^1_x}+\|u^\D( 0)-\psi^\D(0)\|_{L^1_x}\le  C\D 	.
		\end{split}
	\end{equation}
	Moreover, by \eqref{num:trunc_error}, \eqref{num:consistent} and \eqref{num:nearby_2} we have 
	\begin{equation}\label{num:proof3}
		\begin{split}
			\|\partial_t\psi^\D +H(D\psi^\D)\|_{L^1_x}&\le \|\partial_tu^\D +H(Du^\D)\|_{L^1_x}+	\|\partial_t(u^\D-\psi^\D)\|_{L^1_x}\\
			&+L_H\|D(u^\D-\psi^\D)\|_{L^1_x}\le C\D
		\end{split}
	\end{equation}
	Replacing \eqref{num:proof2} and \eqref{num:proof3} in \eqref{num:proof1}, we get
	\eqref{num:est_Lp}.
\end{proof}
Note that, for the estimate \eqref{num:est_Lp}, the family of approximate solutions $\{u^\D\}$ is not required to satisfy a  uniform semiconcavity estimate, as it fails in general for numerical schemes. Instead, it is sufficient to find a family $\{\psi^\D\}$, uniformly semiconcave in $\D$, which is ``close" in a certain sense to that given by the Godunov scheme.
In this context, the key property is the following discrete semiconcavity  introduced in \cite{LinTadmor}: a family ${u^\D}$, $\D>0$, is said semiconcave stable if there exists $k(t)\in L^1(0,T)$ such that for all $h\ge h_0(\D)>0$, there holds
\[
D^2_{h,\xi}u(x,t)\le k(t)\qquad \forall |\xi|=1.
\]
The previous property allows the construction of a family $\{\psi^\D\}$ satisfying   \eqref{num:nearby_1}-\eqref{num:nearby_2} for several examples of Godunov schemes as described in \cite[Section 3]{LinTadmor} to which we refer for further details.

\section{Qualitative properties for Hamilton-Jacobi equations}
\subsection{Regularizing effects for Cauchy problems in the whole space $\R^n$}\label{sec;regeff}
\subsubsection{Hamiltonians satisfying \eqref{H4}}
We consider in this section Lipschitz regularizing effects for  the Hamilton-Jacobi equation
\begin{equation}\label{hjregularizing}
\begin{cases}
\partial_t u_\eps-\eps \Delta u_\eps+H(Du_\eps)=0&\text{ in }\R^n\times(0,\infty),\\
u(x,0)=u_0(x)&\text{ in }\R^n,
\end{cases}
\end{equation}
namely we inquire whether the solution of the Cauchy problem for  the Hamilton-Jacobi equation is smoother, as time evolves, than the initial condition $u_0$. We prove the following
\begin{thm}\label{regLip}
Let $u_0$ be bounded and assume that $H$ satisfies \eqref{H4} with $\widetilde{C}_{H,4}=0$, $\gamma\leq2$ and $f\equiv 0$. Let $u_\eps$ be any bounded classical solution to \eqref{hjregularizing}. Then, for any $\delta>0$ we have $u_\eps\in W^{1,\infty}(\R^n\times(\delta,T))$ and, in particular, the following a priori estimate holds
\[
\|Du(t)\|_{L^\infty(\R^n)}\leq C\|u\|_{L^\infty(Q)}^{\frac{1}{\gamma}}t^{-\frac1\gamma}
\]
for a positive constant $C>0$.
\end{thm}
\begin{proof}
	We drop the subscript $\eps$ for simplicity of notation. By Thereom \ref{semic>0}, we have for all $\tau\in(0,\infty)$
	\[
	u_{ee}\leq \frac{C(\delta +\|Du\|_{\infty}^2)^\frac{2-\gamma}{2}}{t}=:C_0(t).
	\]
The one-side interpolation inequality, cf. \cite{L82book},
	\[
	\|Du(t)\|_{L^\infty(\R^n)}\leq \sqrt{2}\|(D^2u(t))^+\|_{L^\infty(\R^n)}^\frac12\|u(\cdot,t)\|_{L^\infty(\R^n)}^\frac12
	\]
	implies for a.e. $t\in(0,T)$
	\[
	\|Du(t)\|_{L^\infty(\R^n)}\leq (2\|u(t)\|_{L^\infty(\R^n)}C_0(t))^\frac12=\sqrt{2C}\|u\|_{L^\infty(Q)}^\frac12(\delta +\|Du\|_{L^\infty(Q)}^2)^\frac{2-\gamma}{4}t^{-\frac12}.
	\] 
	Rearranging the terms this implies
	\[
	\|Du(t)\|_{L^\infty(\R^n)}\leq 2^\frac1\gamma C^\frac{1}{\gamma}\|u\|_{L^\infty(Q)}^{\frac{1}{\gamma}}t^{-\frac1\gamma}
	\]
	and concludes the proof.
\end{proof}

\begin{rem}\label{Nwave}
Recalling the relation among Hamilton-Jacobi equations and conservation laws, one can exploit the one side interpolation inequality due to E. Tadmor \cite{LinTadmor}, along with the Oleinik one-side Lipschitz condition from Remark \ref{conlaws}, to conclude for the solution $u$ of the conservation law
\[
\partial_t u+(F_i(u))_{x_i}=0
\]
the time-decay in sup-norm for large times due to R. Diperna
\[
\|u(\cdot,t)\|_{\infty}\lesssim t^{-\frac12},
\]
see e.g. Theorem 5 and 6 in Section 3.4 of \cite{EvansBook}, and the references therein.
\end{rem}
\subsubsection{Coercive Hamiltonians satisfying \eqref{H1}}\label{coercive}
In this section we discuss regularizing effects for general coercive Hamiltonians as initiated in \cite{Laurencot,Lions85aa}. The main result is the following\begin{thm}
Let $u_\eps$ be a classical solution of \eqref{hjregularizing} with $H\in W^{1,\infty}_{\mathrm{loc}}$ satisfying \eqref{H1} with $\widetilde{C}_{H,1}=0$. Then
\[
\|Du_\eps(\tau)\|_{L^\infty(\R^n)}\leq \frac{1}{C_{H,1} \tau^\frac1\gamma}\left(\mathrm{osc}_{\R^n\times(0,T)}u_\eps\right)^\frac1\gamma,\ \tau\in(0,T).
\]
\end{thm}
\begin{proof}
We first differentiate the PDE along a unitary direction $e\in\R^n$ to find the equation for $v=(u_\eps)_e$
\[
\partial_t v-\eps \Delta v+D_pH(Du)\cdot Dv=0.
\]
From now on we drop the subscript $\eps$. Then, the function $w=tv$ solves
\[
\partial_t w-\eps\Delta w+D_pH(Du)\cdot Dw=u_e.
\]
By duality we find, using $u_e=Du\cdot e$, the H\"older's inequality and the bound \eqref{Dugamma} (applied with $f=0$ and $\widetilde{C}_{H,1}=0$)
\begin{multline*}
\int_{\R^n}w(\tau)\rho_\tau(x)\,dx\leq \iint_{Q}|u_e|\rho\,dxdt\leq \iint_Q|Du|\rho\,dxdt\\\leq \left(\iint_Q|Du|^\gamma\rho\,dxdt\right)^\frac1\gamma\left(\iint_Q\rho\,dxdt\right)^\frac{1}{\gamma'}
\leq \frac{1}{C_{H,1}}\left(\mathrm{osc}_{\R^n\times(0,T)}u\right)^\frac1\gamma\  \tau^{\frac{1}{\gamma'}}.
\end{multline*}
This implies 
\[
\tau\|Du(\tau)\|_\infty\leq \frac{\tau^{\frac{1}{\gamma'}}}{C_{H,1}}\left(\mathrm{osc}_{\R^n\times(0,T)}u\right)^\frac1\gamma\implies \|Du(\tau)\|_\infty\leq \frac{1}{{C_{H,1}}\tau^\frac1\gamma}\left(\mathrm{osc}_{\R^n\times(0,T)}u\right)^\frac1\gamma.
\]
\end{proof}
\begin{rem}\label{firstorder}
The above result, being independent of the diffusion, applies even to problems with nonlocal diffusion driven by $(-\Delta)^s$, provided that the solution of the dual problem satisfies $\int_{\R^n}\rho(t)\leq1$. This allows to complete some results in \cite{Silvestre}, see Remark 7.3 therein. In particular, it applies to the first-order equation $\partial_t u+H(Du)=0$ after the vanishing viscosity regularization. In the case $H(Du)=\frac{|Du|^\gamma}{\gamma}$, $\gamma>1$, we have $C_H=\frac{1}{\gamma'}$ and find the estimate (independent of the viscosity)
\[
\|Du(\tau)\|_{L^\infty}\leq \frac{(\gamma')^\frac1\gamma}{\tau^\frac1\gamma}(\mathrm{osc}_{\R^n\times(0,T)}u)^\frac1\gamma.
\]
The Hopf-Lax formula shows that the estimate is attained as an equality, as observed in Remark (iii), p. 286, of \cite{Lions85aa}.
\end{rem}
\subsection{Liouville-type theorems for first-order Hamilton-Jacobi equations}\label{sec;lio}
In this section we establish some non-existence properties for generalized (ancient) solutions to the parabolic problem
\begin{equation}\label{ancient}
\partial_t u+H(Du)=0\text{ in }\R^n\times(-\infty,0),
\end{equation}
when $H$ satisfies (H4) with $\gamma\in(1,2]$.  Our main result generalizes the Liouville type property appeared in Section 1 of \cite{K67II}. These results are rather unnatural due to the absence of diffusive terms in the equation. We may assume without loss of generality that $H(0)=0$, otherwise replace $u$ with $\bar u$ solving
\[
\bar u=u-tH(0).
\]
The result reads as follows:
\begin{thm}\label{lio}
Let $u$ be a solution of \eqref{ancient} with $H(0)=0$ and satisfying the one-side decay condition
\[
u(x,t)\geq -|x|\mu(|x|)+K(t),\ \mu(r)\geq0, \mu(r)\to0\text{ as }r\to\infty,
\]
for a bounded function $K$. Then $u$ must be constant.
\end{thm}
\begin{rem}
The previous decay condition is satisfied if $u$ is bounded from below or has   sublinear decay at infinity. In such a case $\mu(r)=r^{\alpha-1}$, $r=|x|$, $\alpha\in(0,1)$, satisfies the condition $\mu(r)\to0\text{ as }r\to\infty$ with $\mu(r)\geq0$.
\end{rem}
To prove the vanishing property, we premise the following lemma of convex analysis taken from Lemma 2 in Section 1 of \cite{K67II}. It extends the classical result saying that a concave function bounded from below must be constant, allowing for a more general unilateral decay condition.
\begin{lemma}
Let $u:\R\to\R$ be a concave function. If 
\[
u(s)\geq -|s|\mu(|s|)+C,\ \mu(r)\geq0,\ \mu(r)\geq0,\ \mu(r)\to0\text{ as }r\to\infty
\]
then $u$ is constant.
\end{lemma}
\begin{proof}[Proof of Theorem \ref{lio}]
The proof relies on showing that in the long-time regime the solution of \eqref{ancient} becomes concave. We can view a solution to \eqref{ancient} as a solution of the problem on the layer $\{-T\leq t\leq0\}$ by letting $T\to\infty$ with initial condition $u(-T)=u(x,-T)$. By Theorem \ref{semic>0} we have
\[
D^2u\,e\cdot e\leq \frac{C}{T+t}.
\]
We fix $x_0,t_0$ and restrict $u$ on the line $u(x_0,t_0+e s)$. Letting $T\to\infty$ in the second-order estimate we conclude that $u$ is a concave function satisfying the growth conditions of the lemma. This implies that $u$ must be constant in space, i.e. $u(x,t)=k(t)$. Then, by the equation we have
\[
\partial_t u=-H(Dk(t))=-H(0)=0,
\]
which implies that $u$ is also constant in the time-variable.
\end{proof}
\begin{rem}
The Liouville-type result for the evolutive problem \eqref{ancient} is determined by the nonlinearity. Indeed, even for the simplest heat equation $\partial_t u-u_{xx}=0$ in $\R^2$ or $\R\times(-\infty,0]$ the Liouville property does not hold for solutions satisfying only one-side bounds: the function $u(x,t)=e^{x+t}$ is caloric, bounded from below and it is not a constant.\\
Furthermore, Remark 2 in \cite{K67II} shows that the lower bound on $u$ cannot be replaced by an upper bound (e.g. $u\leq0$). Other polynomial Liouville theorems can be obtained following the lines of \cite{K67II}.
\end{rem}
\begin{rem}
Since the semiconcavity estimates of Theorem \ref{semic>0} are independent of $\eps>0$, one can prove with the same proof of Theorem \ref{lio} a Liouville theorem for ancient solutions to the following model viscous problem
\[
\partial_t u-\Delta u+|Du|^\gamma=0\quad \text{ in }\R^n\times(-\infty,0).
\]
P. Souplet and Q.S. Zhang \cite[Theorem 3.3]{SZ} proved by the Bernstein method that any classical solution to the above equation such that $|u(x,t)|=o(|x|+|t|^\frac1\gamma)$, $\gamma\in(1,2]$, as $|x|+|t|^\frac1\gamma\to\infty$, must be a constant. In particular, any bounded solution to the above equation is a constant. Our result, instead, requires only a one-side condition in space, but asks an a priori sublinear decay in $x$. Clearly, if $u$ is bounded, Theorem \ref{lio} leads to the same conclusion as that in \cite[Theorem 3.3]{SZ}.
\end{rem}

\appendix

\section{Well-posedness and stability estimates for equations with divergence-type terms}
\subsection{Some useful properties of advection-diffusion equations with bounded drifts}
We consider here some preliminary properties of the Fokker-Planck equation
\begin{equation}\label{adjoint}
	\begin{cases}
		-\partial_{t} \rho - \eps \Delta \rho +\textrm{div} (b(x,t) \rho) = 0, &\text{ in }\R^n\times(0,\tau),
		\\
		\rho(x, \tau) = \rho_{\tau}(x), &\text{ in }\R^n.
	\end{cases}
\end{equation}
We have the following
\begin{lemma}\label{well}
Assume $b\in L^\infty_{\mathrm{loc}}(\R^n\times(0,\tau))$. Then there exists a distributional solution $\rho$ to \eqref{adjoint} verifying $\int_{\R^n}\rho(t)\,dx\leq1$. If in addition $b\in L^2(\R^n\times(0,\tau))$, $|b|\in L^k(\rho\,dxdt)$ for some $k>1$, $\eps>0$, we have a unique weak (energy) solution of \eqref{adjoint} which satisfies $\rho\in L^\infty([0,\tau];L^1(\R^n))$. Moreover, if $\|\rho_\tau\|_1=1$ and $\rho_\tau\geq0$, we have $\|\rho(t)\|_1=1$ for all $t\in[0,\tau)$ and $\rho\geq0$ in $Q$.
\end{lemma}
\begin{proof}
The existence of a distributional solution with a locally bounded drift follows from Theorem 6.2.2 in \cite{BKRS}. Theorem 2.1 and Remark 2.2 in \cite{ManitaShaposhnikov} imply also that $\int_{\R^n}\rho(t)\,dx\leq1$ for all positive times. The equality case (i.e. the conservation of mass) requires a global property on the velocity field $b$. The estimate $\rho\in L^\infty(0,\tau;L^1(\R^n))$ follows by approximating the Cauchy problem in $\R^n$ via a Cauchy-Dirichlet problem on a sequence of expanding domains like $Q_R:=B_R(0)\times(0,\tau)$ and using Proposition 3.10-(i) in \cite{PorrARMA}, which is in turn based on testing the equation against an approximation of the sign function.   The uniqueness statement from Theorem 3.7 of \cite{PorrettaUMI} gives the existence and uniqueness of a weak energy solution.  We now prove the conservation of mass by means of a cut-off argument and this last bound. First, we apply Lemma 9.1.1 in \cite{BKRS} and choose $\chi \in C^{\infty}_{c}(\R^n)$ such that $\chi(x) = 1$ for any $x \in B_1$ and $\chi(x) = 0$ for any $x \in \R^n\backslash B_2$. Then, setting $\chi_R(x) = \chi\left(\frac{x}{R}\right)$ and testing the Fokker-Planck equation against such a function $\chi_R$ via the identity in \cite{BKRS} we get
\[
\int_{\R^n}\rho(t)\chi_R(x)\,dx+\iint_{Q}\eps\rho\Delta \chi_R(x)+\rho b\cdot D\chi_R(x)\,dxdt=\int_{\R^n}\rho(\tau)\chi_R(x)\,dx.
\]
We thus have
\[
\iint_{Q}\rho\Delta \chi_R(x)\,dxdt\leq \frac{1}{R^2}\iint_Q\rho\Delta\chi\,dxdt\leq \frac{C}{R^2}
\]
and
\begin{multline*}
\iint_{Q}\rho b\cdot D\chi_R(x)\,dxdt\leq \frac{1}{R}\iint_{Q}|b|\rho|D\chi|\,dxdt\\
\leq \frac{C}{R}\left(\iint_Q|b|^k\rho\,dxdt\right)^\frac{1}{k}\left(\iint_Q\rho\,dxdt\right)^\frac{1}{k'}\leq \frac{C}{R}\left(\iint_Q|b|^k\rho\,dxdt\right)^\frac{1}{k}\|\rho\|_{L^\infty(0,\tau:L^1(\R^n))}^\frac{1}{k'}. 
\end{multline*}
The conservation of the $L^1$ norm follows then from the dominated convergence theorem and the estimate of $\rho\in L^\infty(0,\tau;L^1(\R^n))$.
\end{proof}
\begin{rem}
When $b=-D_pH(Du_\eps)\in L^\infty_{\mathrm{loc}}(\R^n\times(0,\tau))$, where $u_\eps$ solves \eqref{hjintro}, the global condition $|b|\in L^{k}(\rho\,dxdt)$ with $k=\gamma'$ is satisfied under the assumption \eqref{H1}. Such a requirement is needed for the conservation of mass and is verified in view of Lemma \ref{cross}. Note that Lemma \ref{cross} itself requires to test against the solution $\rho$, but to deduce such a result we only need the existence of densities, cf. Theorem 6.2.2 of \cite{BKRS} valid under the sole assumption of local boundedness of the velocity field $b$, and the condition $\rho\in L^\infty(0,\tau;L^1(\R^n))$. Lemma \ref{cross} does not require the validity of the $L^1$ preserving property.
In case $b\in L^\infty(\R^n\times(0,\tau))$ or in the case of compact domains, the further integrability condition $|b|\in L^{k}(\rho\,dxdt)$ is no longer needed to prove the preservation of $L^1$ norms.
\end{rem}

\subsection{$L^r$ stability estimates}
In this section we consider the Cauchy problem
\begin{equation}\label{fp}
\begin{cases}
\partial_t \rho-(a_{ij}(x,t)\rho)_{x_ix_j}+\mathrm{div}(b(x,t)\rho)=0&\text{ in }Q\\
\rho(x,0)=\rho_0(x)&\text{ in }\R^n.
\end{cases}
\end{equation}
when $\mathrm{div}(b)$ is not divergence-free, and find sufficient conditions that guarantee the validity of $L^p$-stability estimates. 

The next result concerns the degenerate case, cf. Theorem 6.7.4 of \cite{BKRS}.
\begin{thm}\label{Krylov}
Let $q>1$, $b\in L^q(0,T;L^q_{\mathrm{loc}}(\R^n))$ and $a_{ij}\in L^\infty(0,T;W^{1,q}_{\mathrm{loc}}(\R^n))$ with $A\geq0$ (no strict parabolicity is needed). Let $r=\frac{q}{q-1}$ and suppose that
\begin{equation}
\left[(r-1)\left(\mathrm{div}(b)-\sum_{i,j}(a_{ij})_{x_ix_j}\right)\right]^-\in L^1(0,T;L^\infty(\R^n))
\end{equation}
along with $\rho_0\in L^r(\R^n)$. Then, there exists a solution $\rho$ of \eqref{fp} in $L^\infty(0,T;L^r(\R^n))$ and it holds the estimate.
\begin{equation*}
\sup_{t \in [0,T]} \int_{\R^n} |\rho(t, x)|^r\ dx \leq\|\rho_0\|_{L^r(\R^n)}e^{t\|(r-1)[\mathrm{div}(b)-\sum_{i,j}(a_{ij})_{ij}]^-\|_{L^1(0,T;L^\infty(\R^n))}}. 
\end{equation*}
\end{thm}

\end{document}